\documentclass[a4paper,11pt]{article}
\usepackage[latin1]{inputenc}
\usepackage[english]{babel}
\usepackage{amsmath}
\usepackage{amsfonts}
\usepackage{amssymb}
\usepackage{epsfig}
\usepackage{amsopn}
\usepackage{amsthm}
\usepackage{color}
\usepackage{graphicx}
\usepackage{enumerate}
\usepackage{fancyhdr}
\newtheorem{theorem}{Theorem}[section]
\newtheorem{corollary}[theorem]{Corollary}
\newtheorem{lemma}[theorem]{Lemma}
\newtheorem{proposition}[theorem]{Proposition}
\newtheorem{definition}[theorem]{Definition}

\newtheorem*{theorem*}{Theorem}
\newtheorem*{lemma*}{Lemma}
\newtheorem*{remark*}{Remark}
\newtheorem*{definition*}{Definition}
\newtheorem*{proposition*}{Proposition}
\newtheorem*{corollary*}{Corollary}
\numberwithin{equation}{section}
\newcommand{\real}{\mathbb{R}}

\def\a{\alpha}

\def\e{\varepsilon}        
\def\g{\gamma}
\def\h{\eta}

\def\qed{\,\unskip\kern 6pt \penalty 500
\raise -2pt\hbox{\vrule \vbox to8pt{\hrule width 6pt
\vfill\hrule}\vrule}\par}
\definecolor{darkblue}{rgb}{0.05, .05, .65}
\definecolor{darkgreen}{rgb}{0.1, .65, .1}
\definecolor{darkred}{rgb}{0.8,0,0}
\newcommand{\beqn}{\begin{equation}}
\newcommand{\eeqn}{\end{equation}}
\newcommand{\bear}{\begin{eqnarray}}
\newcommand{\eear}{\end{eqnarray}}
\newcommand{\bean}{\begin{eqnarray*}}
\newcommand{\eean}{\end{eqnarray*}}
\begin{document}

\title{\bf Rotationally symmetric $p$-harmonic flows from $D^2$ to $S^2$: local well-posedness and finite time blow-up}

\author{
\Large Razvan Gabriel Iagar\,\footnote{Departamento de Análisis
Matemático, Univ. de Valencia, Dr. Moliner 50, 46100, Burjassot
(Valencia), Spain, \textit{e-mail:}
razvan.iagar@uv.es},\footnote{Institute of Mathematics of the
Romanian Academy, P.O. Box 1-764, RO-014700, Bucharest, Romania.}
\\[4pt] \Large Salvador Moll,\footnote{Departamento de Análisis
Matemático, Univ. de Valencia, Dr. Moliner 50, 46100, Burjassot
(Valencia), Spain, \textit{e-mail:} j.salvador.moll@uv.es}\\ [4pt] }
\date{}
\maketitle

\begin{abstract}
We study the $p$-harmonic flow from the unit disk $D^2$ to the unit
sphere $S^2$ under rotational symmetry. We show that the Dirichlet
problem with constant boundary conditions is locally well-posed in
the class of classical solutions and we also give a sufficient
criterion, in terms of the boundary condition, for the derivative of
the solutions to blow-up in finite time.
\end{abstract}


\medskip


\section{Introduction and main results}

Given a domain $\Omega\subset\real^N$, a real number $p>1$, a
smoothly embedded compact submanifold (without boundary) $M$ of
$\real^{N+1}$, and a mapping ${\bf u}:\Omega\mapsto M$, we consider the
following functional:
\begin{equation}\label{eq1}
E_{p}({\bf u}):=\frac{1}{p}\int_{\Omega}|\nabla {\bf u}|^p\,dx.
\end{equation}
For a generic point ${\bf v}\in M$, we denote by $\pi_{\bf v}$ the
orthogonal projection of $\real^{N+1}$ onto the tangent space
$T_{\bf v}M$ at point ${\bf v}$. With this notation, the
$p$-harmonic flow associated to $E_p$ is given by
\begin{equation}\label{eq2}
{\bf u}_t=-\pi_{{\bf u}}(-\Delta_{p}{\bf u}), \quad \Delta_{p}{\bf u}=\hbox{div}(|\nabla
{\bf u}|^{p-2}\nabla {\bf u}).
\end{equation}
If we particularize $M$ to be the unit sphere $S^N$ of
$\real^{N+1}$, \eqref{eq2} can be written as an explicit parabolic
system of partial differential equations:
\begin{equation}\label{eq3}
{\bf u}_t=\hbox{div}(|\nabla {\bf u}|^{p-2}\nabla {\bf u})+{\bf u}|\nabla {\bf u}|^p.
\end{equation}
Besides its mathematical interest on the results of the competition
between the diffusion term and the reaction in the form of gradient
terms, the system \eqref{eq3} has been also proposed in several
applications: ferromagnetism \cite{dSPG}, theory of liquid crystals
\cite{vdH}, multigrain problems \cite{KWC} and image processing
\cite{SapiroBook}. In this last context, the system \eqref{eq3} has
been used as a prototype for more complicated reaction-diffusion
systems modeling the evolution of director fields.

\medskip

From the mathematical point of view, we deal with the Dirichlet
problem for \eqref{eq3} with (constant) boundary condition ${\bf
u}(t,x)={\bf u}_0(x)$ for $(t,x)\in\partial Q$, where
$Q:=(0,\infty)\times\Omega$. This problem has been widely studied in
the last periods, starting with the more standard case $p=2$, see
for example \cite{BPP, BPvdH}, then for general $p>1$ \cite{chh, FR
Hbook, Misawa}. A special case in the analysis, with some particular
mathematical features due to lack of sufficient regularity, is the
limit case $p=1$, intensively studied recently; indeed, local
well-posedness, steady states and either the appearance of finite
time blow-up or the existence of global solutions were studied for
$p=1$ in  \cite{dPGM, GM, GK}; altogether, it has been shown that
for suitable boundary data, classical solutions exist in short time,
and under some conditions that their first derivative at the origin
blows-up in finite time. In all these papers, the case of
rotationally symmetric solutions and stationary boundary conditions
is considered. On the other hand, existence and some uniqueness
results have been recently obtained in \cite{GMM1, GMM2} for the
Neumann problem and $p=1$.

In a previous work \cite{IM1}, the authors studied  the rotationally symmetric steady states corresponding to the system  for
$1<p<\infty$, $M=S^2$, the unit sphere in $\real^3$, $\Omega=D^2$,
the unit disk in $\real^2$. A rotationally symmetric solution to \eqref{eq3} has the
general form
\begin{equation}\label{rotsim}
{\bf u}(t,x)=\left(\frac{x_1}{r}\sin\,h(t,r),\frac{x_2}{r}\sin\,h(t,r),\cos\,h(t,r)\right),
\quad r=|x|, \ x=(x_1,x_2)\in D^2,
\end{equation}
and the Dirichlet boundary condition $h(t,1)=h(0,1)=l$. The energy
functional $E_p$ in \eqref{eq1} becomes
\begin{equation}\label{pEnergyRotSim}
E_{p}({\bf u})=\int_0^1 L_p(r,h,h_r)\,dr,
\end{equation}
with a non-coercive Lagrangian given by
$$L_p(r,s,\xi)=\left(r^{2/p}h_r^2+r^{(2-2p)/p}\sin^2\,h\right)^{p/2}.$$
By replacing ${\bf u}$ in \eqref{eq3} by its special form \eqref{rotsim},
the system \eqref{eq3} becomes
\begin{equation}\label{pHarmRotSim}
\begin{split}
h_t&=\left(h_r^2+\frac{sin^2\,h}{r^2}\right)^{(p-4)/2}\left[(p-1)h_r^2h_{rr}+(p-3)\left(\frac{h_r^2\sin\,h\cos\,h}{r^2}-\frac{h_r\sin^2\,h}{r^3}\right)\right.\\
&+\left.\frac{h_r^3}{r}+\frac{h_{rr}\sin^2\,h}{r^2}-\frac{\sin^3\,h\cos\,h}{r^4}\right],
\end{split}
\end{equation}
see \cite{IM1} for details. In our previous work, we studied and
classified the steady states for \eqref{pHarmRotSim}, for all
$1<p<\infty$.

In the present paper we continue the study of the rotationally
symmetric flows begun in \cite{IM1}, and we devote it to some
qualitative properties of the solution ${\bf u}$ to \eqref{eq3}, or
equivalently, $h$ solution to \eqref{pHarmRotSim}, such as existence
and well-posedness for short times, regularity of solutions and
conditions for the phenomenon of blow-up in finite time to occur. We
restrict ourselves to the fast diffusion range $1<p<2$, since, for
the range $p\geq 2$, it has been shown in \cite{FR} that strong
solutions exist globally for any initial datum, they are H\"older
continuous and they converge to a steady-state as $t\to +\infty$. A
similar result for the range $1<p<2$ was also obtained in \cite{FR}
but with a condition of small energy on the initial datum. We point
out that our approach and results are totally different to those in
\cite{FR}. First of all, we obtain local existence of classical
solutions for any initial datum and secondly, we give a sufficient
condition on the boundary constraint for classical solutions ceasing
to exist in finite time.

\medskip
For convenience, we denote
$$
J(h):=h_r^2+\frac{\sin^2\,h}{r^2},
$$
and
$$
A(h):=(p-1)h_r^2h_{rr}+(p-3)\left(\frac{h_r^2\sin\,h\cos\,h}{r^2}-\frac{h_r\sin^2\,h}{r^3}\right)
+\frac{h_r^3}{r}+\frac{h_{rr}\sin^2\,h}{r^2}-\frac{\sin^3\,h\cos\,h}{r^4},
$$
thus \eqref{pHarmRotSim} together with the initial and boundary
conditions can be written in the form of the following Dirichlet
problem:
\begin{equation}\label{pHarmDirichlet}
\left\{\begin{array}{ll}h_t=F(h):=J(h)^{(p-4)/2}A(h), \quad
(t,r)\in(0,T)\times(0,1),\\h(0,r)=h_0(r), \quad r\in(0,1), \\
h(t,1)=h_0(1), \quad t\in[0,T).\end{array}\right.
\end{equation}
We introduce the concept of solution that we will use throughout the
paper.
\begin{definition}\label{DefSol}
A classical solution to the Dirichlet problem \eqref{pHarmDirichlet}
in $(0,T)$ with initial datum $h_0$ is a function $h\in
C^{1,2}([0,T]\times[0,1])$ which satisfies \eqref{pHarmDirichlet}
pointwisely in $[0,T]\times[0,1]$.
\end{definition}
Arguing similarly as in the proof of \cite[Lemma 2.1]{GK}, one can
easily get that ${\bf u}$ is a rotationally symmetric solution to
the Dirichlet problem associated to the system \eqref{eq3} if and
only if ${\bf u}$ has the form \eqref{rotsim} with $h$ being a
classical solution to the Dirichlet problem \eqref{pHarmDirichlet}
such that $h(t,0)=0$ for all $t\in[0,T)$. Our study is divided into
two main parts: a first one devoted to the local well-posedness of
the Dirichlet problem \eqref{pHarmDirichlet}, and a second, more
specialized one dealing with the phenomenon of blow-up in finite
time. We point out that the results in this paper and the strategy
to prove them are analogous to those of the case $p=1$ provided in
\cite{GM}. However, due to the special nonlinear term
$(p-1)h_r^2h_{rr}$ in \eqref{pHarmRotSim}, all the proofs are in the
present case much more involved than those in \cite{GM}.

\medskip

\noindent \textbf{Local well-posedness.}  We notice that, if
$h(t,r)=k\pi$ for some $r\in(0,1]$ and $k\in\mathbb{Z}$, then the
Dirichlet problem \eqref{pHarmDirichlet} becomes degenerate around
$(t,r)$. In order to avoid this extra difficulty, we assume that the
initial condition $h_0$ satisfies the following conditions
\begin{equation}\label{initcond}
h_0\in C^2([0,1]), \quad h_0(0)=0, \quad h_0(r)\in(0,\pi), \
\hbox{for} \ \hbox{all} \ r\in(0,1], \quad h_0'(0)>0.
\end{equation}
In this way, we avoid the degeneracy at $t=0$. However, these
properties are preserved along the evolution. More precisely, we
have the following uniqueness result.
\begin{theorem}\label{th.unique}
Let $p\in(1,2)$ and assume that $h_0$ satisfies \eqref{initcond}.
Then there exists at most one classical solution $h$ defined in a
maximal interval $(0,T)$ to the Dirichlet problem
\eqref{pHarmDirichlet} with initial condition $h_0$. Moreover, the
solution $h$ satisfies the following

\noindent (a) There exists $\delta=\delta(h_0)\in(0,1)$, such that
\begin{equation}\label{uniq1}
h(t,0)=0, \quad \delta r\leq h(t,r)\leq\pi-\delta r, \quad
h_r(t,0)\geq\delta,
\end{equation}
for all $r\in(0,1]$, $t\in (0,T)$.

\noindent (b) Furthermore,
\begin{equation}\label{uniq2}
h_{rr}(t,0)=0, \quad F(h(t,\cdot))|_{r=1}=0,
\end{equation}
for all $t\geq0$.
\end{theorem}
The proof of Theorem \ref{th.unique} is given in Subsection
\ref{subsec.unique} and follows as a consequence of a more general
comparison principle that will be proved in Section \ref{sec2} and
of the availability of explicit sub- and supersolutions.

In order to prove existence of solutions and continuous dependence
on the initial data, we need to impose some further conditions on
the initial datum.
\begin{theorem}\label{th.exist}
Let $h_0$ be a function that satisfies \eqref{initcond} and,
\begin{equation}\label{initcond2}
h_{0rr}(0)=0, \quad F(h_0)|_{r=1}=0, \quad
r\mapsto\frac{h_0(r)}{r}\in C^2([0,1]).
\end{equation}
Then there exists a classical solution $h$ to \eqref{pHarmDirichlet}
with initial condition $h_0$, defined in a maximal interval $[0,T)$.
Moreover, the problem is well-posed (that is, the solution depends
continuously on $h_0$) and we have further regularity results:

(a) $$(t,r)\mapsto\frac{h(t,r)}{r}\in C^{1,2}([0,T)\times[0,1]),$$

(b) $$h_r\in C^{1,2}((0,T]\times(0,1]).$$
\end{theorem}
The proof is quite complex and relies on the application of a
general theory of sectorial operators, that is developed in an
abstract framework in the book \cite{Lunardi}, after noticing that
the linearization of the (fully nonlinear) elliptic operator
associated to our equation is indeed sectorial. Precise statements
and proofs are given in Subsection \ref{subsec.lex}.

\medskip

\noindent \textbf{Finite time blow-up.} This is the more specialized
part of the paper, where we prove that, if $h_0(1)$ is sufficiently
large, then the solution to \eqref{pHarmDirichlet} is not global;
more precisely, its first derivative blows up in finite time. The
critical value for $h_0(1)$ is related to the non-existence of
stationary solutions. The stationary solutions to
\eqref{pHarmRotSim} for $p\in(1,2)$, that is, solutions to $A(h)=0$
in our notations, were completely classified by the authors in
\cite[Theorem 2]{IM1}, which we recall below in a shortened version
which is useful for our aims.
\begin{proposition}\label{prop.stat}
There exists a unique global stationary solution $h_*\in
C^{\infty}((0,\infty))\cap C^1([0,\infty))$ to Eq.
\eqref{pHarmRotSim} (that is, $A(h_*)=0$), which satisfies the
following properties:

\noindent (a) $h_*(0)=0$, $\lim\limits_{r\to\infty}h_*(r)=\pi/2$;

\noindent (b) There exists an increasing sequence of critical points
$r_n$ of $h$ such that $r_n\to\infty$ as $n\to\infty$,
$h_*(r_{2n})\in(\pi/2,\pi)$ and are local maxima,
$h_*(r_{2n+1})\in(0,\pi/2)$ and are local minima. Moreover, there
exists some positive integer $n_0$ sufficiently large such that
$|h_*(r_n)-\pi/2|$ is decreasing for $n\geq n_0$, and
\begin{equation*}
\left|h_*(r_0)-\frac{\pi}{2}\right|=\max\left\{\left|h_*(r_n)-\frac{\pi}{2}\right|:n\geq
0\right\};
\end{equation*}

\noindent (c) Any non-constant finite stationary global solution $h$
to Eq. \eqref{pHarmRotSim} has the form $h(r)=k\pi\pm h_*(\a r)$,
for some $\a>0$ and $k\in \mathbb Z$.
\end{proposition}


Let $H:=h_*(r_0)=\max\{h_*(r):r>0\}$. Then $H\in(\pi/2,\pi)$ and,
from the above proposition, for any $l\in(H,\pi)$, there exists no
stationary solution $\tilde{h}$ to \eqref{pHarmRotSim} satisfying
$\tilde{h}(0)=0$ and $\tilde{h}(1)=l$. Moreover, we have the
following obvious consequence of Proposition \ref{prop.stat}.
\begin{corollary}\label{cor}
For $l\in(H,\pi)$, the unique stationary solution $h$ to
\eqref{pHarmRotSim} such that $h(1)=l$ is given by
$H_l(r)=\pi-h_{*}(\a r)$, where $\a>0$ is determined by
$h_{*}(\a)=\pi-l$.
\end{corollary}
On the other hand, if $h$ is a global solution to the Dirichlet
problem \eqref{pHarmDirichlet}, then it is expected, as usual for
parabolic equations, that its large-time behavior as $t\to\infty$,
to be a stationary state to the equation. Thus, we deduce that if
the initial datum satisfies $h(0)=0$ and $h(1)=l\in(H,\pi)$, the
solution to \eqref{pHarmDirichlet} with initial condition $h_0$ is
not global. Indeed, we have
\begin{theorem}\label{th.blowup1}
Let $h_0$ be as in Theorem \ref{th.exist} and $T>0$ its maximum
existence time. If $h_0(1)\in (H,\pi)$, then $T<\infty$ and the
derivative of the solution $h$ blows up at the origin at $t=T$.
\end{theorem}
The proof of Theorem \ref{th.blowup1} is very involved and technical
and requires several steps that are sketched below.

\noindent (A) We first show that indeed, if a global solution $h$ to
\eqref{pHarmDirichlet} exists, then its behavior as $t\to\infty$ is
given by a stationary state to \eqref{pHarmRotSim}. In order to
prove it, we need part (B) below.

\noindent (B) To prove (A), we first derive some weighted bounds in
$L^{\infty}$ of the derivative $h_r$, more precisely, uniform bounds
of $rh_r$, that in particular imply local boundedness for $h_r$
outside the origin. This is done in Section \ref{sec.bound}.

\noindent (C) We also need some special, explicit subsolutions that
blow up in finite time. Their form proves to be quite simple, but
proving that indeed they are subsolutions to \eqref{pHarmRotSim}
showed to be technically difficult. In fact, we have the following
result proved in the Appendix:
\begin{proposition}\label{prop.subs}
Let $p\in(1,2)$. There exists $\delta_0>0$, such that
$$
h(t,r):=\frac{p+4}{3}\arctan\left(\frac{r}{b(t)}\right), \quad
b(t)=b_0-t\delta_0,
$$
is a subsolution to Eq. \eqref{pHarmRotSim}, that is
\begin{equation}\label{spec.subs}
h_t\leq J(h)^{(p-4)/2}A(h), \quad {\rm for} \ {\rm any} \
(t,r)\in(0,b_0/\delta)\times(0,1),
\end{equation}
for any $b_0>0$ and $\delta\in(0,\delta_0)$.
\end{proposition}

With all these steps, the proof of Theorem \ref{th.blowup1} follows
easily from a contradiction argument. The proof of claim (A) above
together with this final contradiction argument may be found in
Section \ref{sec.blowup}.

Finally, we can complete the panorama of blow up in finite time by
showing that it may also occur, although not as a general
phenomenon, even when $h(1)<H$. More precisely
\begin{theorem}\label{th.blowup2}
For any $l\in(0,\pi)$, there exist $h_0$, $h$ and $T$ as in Theorem
\ref{th.exist} such that $h_0(1)=l$ and $T<\infty$.
\end{theorem}
The proof of this theorem is similar to the one of \cite[Proposition
2.2]{BPvdH} or \cite[Theorem 4]{GM} and it will be sketched at the
end of Section \ref{sec.blowup}.

\section{Comparison principle}\label{sec2}

In this section we prove a comparison principle for
\eqref{pHarmRotSim} that will be often useful in the sequel. It is
analogous, but technically more involved, to Lemma 4.2 in \cite{GK},
where the comparison principle is proved for $p=1$. To this end, we
establish the notions of sub- and supersolution to
\eqref{pHarmRotSim}
\begin{definition}\label{def.subsuper}
A function $h\in C^{1,2}([0,T]\times[0,1])$ is called a subsolution
(resp. a supersolution) to Eq. \eqref{pHarmRotSim} if
$$
h_t\leq J(h)^{(p-4)/2}A(h), \quad ({\rm resp.} \quad h_t\geq
J(h)^{(p-4)/2}A(h)), \ {\rm for} \ {\rm all} \
(t,r)\in(0,T)\times(0,1).
$$
\end{definition}
We then prove
\begin{proposition}\label{prop.comp}
Assume that $h$, $f\in C^{1,2}([0,T]\times[0,1])$ are two functions
that satisfy $h(t,0)=f(t,0)=0$ for any $t\in[0,T]$ and that
$$
inf\{J(\theta h+(1-\theta)f):
(t,r,\theta)\in(0,T)\times(0,1)\times(0,1)\}>0.
$$
Assume that $h$ is a supersolution to \eqref{pHarmRotSim} and $f$ is
a subsolution to \eqref{pHarmRotSim} and that
$$
h(t,1)\geq f(t,1), \quad h(0,r)\geq f(0,r),
$$
for all $t\in[0,T]$, respectively $r\in[0,1]$. Then $h\geq f$ in
$[0,T]\times[0,1]$.
\end{proposition}
\begin{proof}
Set
\begin{equation*}
\begin{split}
F(r,x,y,z)&=\left(y^2+\frac{\sin^2x}{r^2}\right)^{(p-4)/2}\left\{(p-1)y^2z+(p-3)\left[\frac{y^2\sin
x\cos
x}{r^2}-\frac{y\sin^2x}{r^3}\right]\right.\\&\left.+\frac{y^3}{r}+\frac{z\sin^2x}{r^2}-\frac{\sin^3x\cos
x}{r^4}\right\}
\end{split}
\end{equation*}
and notice that
$$
h_t\geq F(r,h,h_r,h_{rr}), \quad f_t\leq F(r,f,f_r,f_{rr}).
$$
Let $w:=h-f$. Denoting
$$
v(t,r,\theta):=\theta h(t,r)+(1-\theta)f(t,r), \quad
(t,r,\theta)\in(0,T)\times(0,1)\times(0,1),
$$
it follows by the mean-value theorem that $w$ satisfies the
following inequality
\begin{equation}\label{interm1}
w_t-A(t,r)w_{rr}-B(t,r)w_r-C(t,r)w\geq0, \quad
(t,r)\in(0,T)\times(0,1),
\end{equation}
where
$$
A(t,r):=\int_0^1F_{z}\,d\theta, \quad
B(t,r):=\int_0^1F_{y}\,d\theta, \quad
C(t,r):=\int_0^1F_{x}\,d\theta,
$$
in all of these, the integral being taken for $F(r,v,v_r,v_{rr})$.
Moreover, $w(t,0)=0$ and $w(t,1)\geq0$ for all $t\in[0,T)$ and
$w(0,r)\geq0$ for all $r\in[0,1]$. Thus, we have to prove that
\eqref{interm1} satisfies the standard comparison principle for
linear parabolic equations. To this end, we show that $A(t,r)\geq0$
and $C(t,r)$ is bounded from above, for all
$(t,r)\in(0,T)\times(0,1)$. We have
$$
F_z(r,x,y,z)=\left(y^2+\frac{\sin^2x}{r^2}\right)^{(p-4)/2}\left((p-1)y^2+\frac{\sin^2x}{r^2}\right)\geq0,
$$
thus $A(t,r)\geq0$ for all $(t,r)\in(0,T)\times(0,1)$. On the other
hand, after some calculations we find
\begin{equation}\label{interm2}
F_{x}(r,x,y,z)=\left(y^2+\frac{\sin^2x}{r^2}\right)^{(p-6)/2}(P+Q+R),
\end{equation}
where
$$
P(r,x,y,z):=\frac{\sin2x}{r^2}\left[\frac{(p-2)(p-3)}{2}y^2+\frac{p-2}{2}\frac{\sin^2x}{r^2}\right]z,
$$
$$
Q(r,x,y):=\frac{(p-4)\sin2x}{2r^2}\left[(p-3)\left(\frac{y^2\sin
x\cos
x}{r^2}-\frac{y\sin^2x}{r^3}\right)+\frac{y^3}{r}-\frac{\sin^3x\cos
x}{r^4}\right],
$$
and
$$
R(r,x,y):=\left(y^2+\frac{\sin^2x}{r^2}\right)\left[(p-3)\left(\frac{y^2\cos2x}{r^2}-\frac{y\sin2x}{r^3}\right)-\frac{3\sin^2x\cos^2x-\sin^4x}{r^4}\right].
$$
Since $J(v)>0$ up to $r>0$ by hypothesis and $h_r>0$ near $r=0$, we
have
$$
v_r(t_0,r,\theta_0)>0, \quad \cos
v(t_0,0,\theta_0)=\cos2v(t_0,0,\theta_0)=1
$$
for any $(t_0,\theta_0)\in[0,T]\times[0,1]$ and $r$ sufficiently
close to 0. Thus, fixing $a\in(0,1)$ and any
$(t_0,\theta_0)\in[0,T]\times[0,1]$, there exists some $r_a>0$
sufficiently small, that can be chosen independent of
$(t_0,\theta_0)$, such that
\begin{equation}\label{interm4}
y=v_r(t,r,\theta)>0, \quad a<\cos v(t,r,\theta)\leq1, \quad a<\cos
2v(t,r,\theta)\leq1,
\end{equation}
for all $(t,r,\theta)\in W_{a}(t_0,\theta_0)$, where
$W_a(t_0,\theta_0)$ is the following neighborhood of
$(t_0,\theta_0)$:
$$
W_a(t_0,\theta_0):=\{(t,r,\theta)\in[0,T]\times(0,1)\times[0,1]:
r\in(0,r_a), \ |t-t_0|<r_a, \ |\theta-\theta_0|<r_a\}.
$$
Similarly as in \cite{GK}, we denote $m:=\min\limits_{W_a}v_r>0$,
$M:=\max\limits_{W_a}v_r$ and choose $r_a$ small enough such that
$m\geq aM$. Then, by the mean-value theorem, we have
\begin{equation}\label{interm3}
ma\leq\frac{\sin x}{r}\leq M, \quad 2ma^2\leq\frac{\sin2x}{r}\leq2M
\quad {\rm in} \ W_a.
\end{equation}
We can thus estimate the three terms $P$, $Q$, $R$ in $W_a$. First
of all, we get the following estimate for $P$:
\begin{equation}\label{estP}
P(r,x,y,z)\leq\frac{1}{r}M_1(a), \quad M_1(a)>0.
\end{equation}
Indeed, by simply bounding from above all the terms and taking
modulus, using \eqref{interm3} we can take
$$
M_1(a)=2M\left(\frac{(p-2)(p-3)}{2}M^2+\frac{2-p}{2}M^2\right)\max\limits_{W_a}|v_{rr}(t,r,\theta)|>0
$$
and derive \eqref{estP}. In order to estimate $Q$ and $R$, we state
and prove the following

\medskip

\noindent \textbf{Claim:} There exists $M_2(a)$, $M_3(a)$ such that
\begin{equation}\label{estQR}
Q(r,x,y)\leq\frac{1}{r^2}M_2(a), \quad
R(r,x,y)\leq\frac{1}{r^2}M_3(a)
\end{equation}
in $W_a$, with $M_2(a)+M_3(a)<0$ for $a\in(0,1)$ sufficiently close
to 1.

\medskip

\noindent \textbf{End of the proof.} Assume for a moment that the
claim is true. Then, from \eqref{interm2}, \eqref{estP} and
\eqref{estQR} we deduce that $F_{x}(r,x,y,z)<0$ in $W_a$ provided
$r\in[0,-(M_2(a)+M_3(a))/M_1(a))$. Since the only singularity is at
$r=0$, by continuity we get that $F_x$ is bounded from above for any
$r\in(0,1)$ and $(t,\theta)$ such that $|t-t_0|<r_a$ and
$|\theta-\theta_0|<r_a$. Recalling that $r_a$ was taken independent
of the starting point $(t_0,\theta_0)$, we conclude that $F_x$,
whence also $C(t,r)$, is bounded from above in $(0,T)\times(0,1)$.
Then the maximum principle applies to \eqref{interm1} and we reach
the conclusion.

\medskip

\noindent \textbf{Proof of the claim.} To estimate $Q(r,x,y)$, using
\eqref{interm3}, we notice that
\begin{equation*}
\frac{y^2\sin x\cos
x}{r^2}-\frac{y\sin^2x}{r^3}=\frac{1}{r}\left[y^2\frac{\sin2x}{2r}-y\left(\frac{\sin
x}{r}\right)^2\right]\leq\frac{1}{r}(M^3-a^2m^3),
\end{equation*}
hence we find
\begin{equation}\label{interm5}
Q(r,x,y)\frac{2r^2}{(p-4)\sin2x}\geq\frac{1}{r}\left[(p-3)(M^3-a^2m^3)+m^3-M^3\right].
\end{equation}
Since $a<1$, $p\in(1,2)$ and obviously $m<M$, the right-hand side of
\eqref{interm5} is negative. Thus, using once more \eqref{interm3},
we finally get
\begin{equation}\label{interm6}
Q(r,x,y)\leq\frac{1}{r^2}M_2(a), \quad
M_2(a)=\frac{M(p-4)}{2}\left[(p-3)(M^3-a^2m^3)+m^3-M^3\right]>0.
\end{equation}
In order to estimate $R$, we first notice that \eqref{interm3} and
\eqref{interm4} imply
$$
\frac{y^2\cos2x}{r^2}-\frac{y\sin2x}{r^3}\geq\frac{am^2-2M^2}{r^2},
$$
hence
\begin{equation*}
\left(y^2+\frac{\sin^2x}{r^2}\right)^{-1}R(r,x,y)\leq\frac{1}{r^2}\left[(p-3)(am^2-2M^2)-a^3m^2-2a^{4}m^2\right].
\end{equation*}
It follows that
\begin{equation}\label{interm7}
R(r,x,y)\leq\frac{1}{r^2}M_3(a), \quad
M_3(a)=(1+a^2)m^2\left[(p-3)(am^2-2M^2)-a^3m^2-2a^{4}m^2\right].
\end{equation}
We notice that, for $a$ sufficiently close to 1, $M_3(a)<0$.
Moreover, from \eqref{interm6} and \eqref{interm7}, and taking into
account that $aM\leq m\leq M$, we have
\begin{equation*}
\begin{split}
M_2(a)+M_3(a)&=\frac{(p-4)M}{2}\left[(p-3)(M^3-a^2m^3)+m^3-M^3\right]\\&+(1+a^2)m^2\left[(p-3)(am^2-2M^2)-a^3m^2-2a^4m^2\right]\\
&\leq\frac{(p-4)M}{2}\left[(p-4)M^3+(a^2(3-p)+1)a^3M^3\right]\\&+a^2(1+a^2)M^2\left[2(3-p)M^2-a^2(a(3-p)+a^3+{2a^2})M^2\right]\\
&=M^4G(a,p),
\end{split}
\end{equation*}
where it is immediate to see that
$$
\lim\limits_{a\to1}G(a,p)=-2p<0.
$$
Thus, for $a$ sufficiently close to 1 we get that $M_2(a)+M_3(a)<0$
and the claim is proved.
\end{proof}

\section{Local well-posedness}\label{sec.lwp}

In this section we show  that the Dirichlet problem
\eqref{pHarmDirichlet} is well-posed at least locally when the
initial condition $h_0$ satisfies \eqref{initcond}.

\subsection{Uniqueness. Proof of Theorem
\ref{th.unique}}\label{subsec.unique}

We begin with the following easy fact.
\begin{lemma}\label{lem.subs}
For all $\lambda>0$, the function
$\Phi_{\lambda}(r):=2\arctan(\lambda r)$ is a subsolution to
\eqref{pHarmRotSim} and the function
$\Psi_{\lambda}(r):=\pi-\Phi_{\lambda}(r)$ is a supersolution to
\eqref{pHarmRotSim}.
\end{lemma}
\begin{proof}
We note that
$$
\Phi_{\lambda}'(r)=\frac{2\lambda}{1+\lambda^2r^2}, \quad
\Phi_{\lambda}''(r)=-\frac{4\lambda^3r}{(1+\lambda^2r^2)^2},
$$
and
$$
\sin\Phi_{\lambda}(r)=\frac{2\lambda r}{1+\lambda^2r^2}, \quad
\cos\Phi_{\lambda}(r)=\frac{1-\lambda^2r^2}{1+\lambda^2r^2}.
$$
After straightforward calculations, we find that
$$
A(\Phi_{\lambda}(r))=\frac{32(2-p)\lambda^5r}{(1+\lambda^2r^2)^2}>0,
$$
since $1<p<2$. In a similar way $\Psi_{\lambda}$ is a supersolution,
we omit the details.
\end{proof}
\begin{proof}[Proof of Theorem \ref{th.unique}] The uniqueness
follows directly from Proposition \ref{prop.comp}. Let then $h$ be a
classical solution with initial condition $h_0$ satisfying
\eqref{initcond}.

\noindent \textbf{Step 1.} We show first that $h(t,0)=0$ for all
$t\in(0,T)$. We argue by contradiction and assume that this is not
true; hence, by continuity there exists $t_0\in(0,T)$ such that
$h(t_0,0)\in(-\pi/2,\pi/2)\setminus\{0\}$. We then deduce from the
equation \eqref{pHarmRotSim} that, as $t=t_0$ and $r\to0^+$,
\begin{equation*}
\begin{split}
h_t&=\left(h_r^2+\frac{\sin^2h}{r^2}\right)^{(p-4)/2}\left[h_{rr}\left((p-1)h_r^2+\frac{\sin^2h}{r^2}\right)\right.\\
&\left.+\frac{h_r}{r}\left(h_r^2+(3-p)\frac{\sin^2h}{r^2}\right)-\frac{\sin
h\cos h}{r^2}\left((3-p)h_r^2+\frac{\sin^2h}{r^2}\right)\right]\\
&=-\frac{\cos h}{r^p}(1+o(1))\quad {\rm as \ } r\to 0^+,
\end{split}
\end{equation*}
which contradicts the regularity of $h$ as a classical solution.
Hence $h(t,0)=0$ for any $t\in(0,T)$.

\medskip

\noindent \textbf{Step 2.} We prove the rest of part (a) in Theorem
\ref{th.unique}. By assumptions in \eqref{initcond}, there exists
$T_1>0$ such that $h_r(t,0)>0$ and $h(t,r)\in(0,\pi)$ for all
$(t,r)\in(0,T_1)\times(0,1]$. Thus, we apply Proposition
\ref{prop.comp} in $(0,T_1)\times(0,1]$ coupled with Lemma
\ref{lem.subs} with $\lambda$ sufficiently small, to get
$$
\Phi_{\lambda}(r)\leq h(t,r)\leq\Psi_{\lambda}(r), \quad (t,r)\in
(0,T_1)\times(0,1].
$$
But since $h(T_1,0)=0$, we notice that
$$
h_r(T_1,0)=\lim\limits_{r\to0^+}\frac{h(T_1,r)}{r}\geq\lim\limits_{r\to0^+}\frac{2\arctan(\lambda
r)}{r}=2\lambda>0,
$$
thus we can restart the argument from $t=T_1$ and extend it up to
$T$ by a standard maximality argument. Finally, since
$$
\frac{\arctan(x)}{x}\geq\frac{\pi}{4}, \quad \hbox{for} \ \hbox{all}
\ x\in(0,1],
$$
we can replace $\Phi_{\lambda}(r)$ by $\delta r$ and
$\Psi_{\lambda}(r)$ by $\pi-\delta r$ in the above inequality,
provided $2\delta\leq\pi\lambda$. In particular, we also get that
$h_r(t,0)\geq\delta$, for all $t\in[0,T)$.

\medskip

\noindent \textbf{Step 3.} We prove part (b) in Theorem
\ref{th.unique}. Since $h(t,0)=0$ for all $t>0$, we have
$h_t(t,0)=0$, hence $A(h)=0$ at $(t,0)$. We evaluate $A(h)$ at
$(t,r)$ as $r\to0^+$, taking into account that,
$$
\frac{h_r^2\sin h\cos h}{r^2}\sim\frac{h_r^3}{r}, \quad
\frac{h_r\sin^2h}{r^3}\sim\frac{h_r^3}{r}, \quad {\rm as} \ r\to0^+
$$
It follows that the singular terms in the expression of $A(h)$
cancel (in a first order approximation) as $r\to0^+$, whence
$ph_{rr}(t,0)h_r^{2}(t,0)=0$ which implies the first part of the
conclusion. The second one is trivial, since
$$
0=h_t(t,1)=F(h(t,\cdot))|_{r=1}.
$$
\end{proof}

\subsection{Local existence and continuous
dependence}\label{subsec.lex}

As explained in the Introduction, Eq. \eqref{pHarmRotSim} has
singular coefficients, thus the standard existence theory for ODEs
cannot be applied directly. In order to simplify its writing and, in
a first step, transform its trigonometric nonlinearities into
algebraic ones, we introduce the following change of variables:
$$
h(t,r)=2\arctan(ru(t,r)).
$$
The equation
satisfied by $u$ reads as follows:
\begin{equation*}
\begin{split}
\frac{2j^{(4-p)/2}}{1+r^2u^2}u_t&=\frac{8}{(r^2u^2 + 1)^4}\left[(ru_{rr}(r^2(p - 1)u_r^2 + 2r(p - 1)uu_r + pu^2)(r^2u^2 + 1)\right. \\ & + 2r^5(1 - p)uu_r^4 - r^2u_r^3(r^2(6p - 7)u^2 - 2p + 1) \\ & \left.- ruu_r^2(3r^2(3p - 4)u^2 - 5p + 4) + u_ru^2(3p - r^2(9p - 16)u_r^2) + 4r(2 - p)u^5)\right],
\end{split}
\end{equation*}
with
 $\displaystyle
j(r,u,w)=\frac{4(u+rw)^2+4u^2}{(1+r^2u^2)^2}
$.
The above equation can be written in a
simplified form as
\begin{equation}\label{eq.transf}
u_t=\left[a(r,u,u_r)u_{rr}+\frac{b(r,u,u_r)}{r}u_r+f(r,u,u_r)\right]=:G(u),
\end{equation}
with
\begin{equation}\label{eq.a}
a(r,u,w)=j(r,u,w)^{(p-4)/2}\frac{4(p-1)(u+rw)^2+4u^2}{(1+r^2u^2)^2},
\end{equation}
\begin{equation}\label{eq.b}
b(r,u,w)=j(r,u,w)^{(p-4)/2}\frac{12pu^2}{(1+r^2u^2)^3},
\end{equation}
and $f(r,u,w)$ gathers the remaining terms in the equation above, we
omit its precise expression.
The initial and Dirichlet conditions become
\begin{equation}\label{eq.transfinit}
u_r(t,0)=0, \ u(t,1)=u_0(1), \quad {\rm for} \ t>0, \quad
u(0,r)=u_0(r), \quad {\rm for} \ r\in[0,1].
\end{equation}
We then follow the ideas in \cite[Section 4]{GM} and consider the
linearization of the right-hand side above around compatible initial
data $u_0$ such that $u_0>0$ and $u_{0,r}(0)=0$. To this end, we
consider the Fr\'echet derivative of the right-hand side of
\eqref{eq.transf} around a generic point $(u_0,u_{0,r})$ and obtain
that the linearized operator has the form:
\begin{equation}\label{linearized}
L(v):=a(r,u_0,u_{0,r})v''+\frac{b(r,u_0,u_{0,r})}{r}v'+c(r,u_0,u_{0,r})v'+d(r,u_0,u_{0,r})v,
\end{equation}
where $a$, $b$ are given respectively in
\eqref{eq.a} and \eqref{eq.b} and
$$
c=\frac{\partial a}{\partial w}u_{0,rr}+\frac{\partial b}{\partial w}\frac{u_{0,r}}{r}+\frac{\partial f}{\partial w}
$$
and
$$
d=\frac{\partial a}{\partial u}u_{0,rr}+\frac{\partial b}{\partial u}\frac{u_{0,r}}{r}+\frac{\partial f}{\partial u}.
$$We note that no singular terms near $r=0$ appear in $c$ and $d$ since $\frac{u_{0,r}}{r}\stackrel{r\to 0^+}\sim u_{0,rr}$.
Without loss of generality, we will assume that $a(r,u_0,u_{0,r})|_{r=0}=1$. We define next the following operators:
\begin{equation}\label{linearized.0}
L_0v:=v''+\frac{3}{r}v',
\end{equation}
and
\begin{equation}\label{linearized.1}
L_{\e}v:=a(r)v''+\frac{b(r)}{r+\e}v'+c(r)v'+d(r)v,
\end{equation}
where we have omitted from the expressions of $a$, $b$, $c$ and $d$
the dependence on $u_0$ for simplicity. From the previous
considerations, we formally expect that
\begin{equation}\label{approx.L}
Lv\sim L_0v, \ \hbox{as} \ r\to 0^+, \quad Lv\sim L_{\e}v, \
\hbox{for} \ r>r_0>0.
\end{equation}
Following \cite[Section 4]{GM}, we will employ the theory of
sectorial operators. We do not recall the definition of a sectorial
operator, that can be found in \cite[Definition 2.0.1]{Lunardi} or
\cite[Definition 4.1]{GM}, but we state the following sufficient
condition that we use in the sequel.
\begin{proposition}\label{prop.sector}
Let $X$ be a complex Banach space and $A:D(A)\subseteq X\mapsto X$
be a linear operator. If there exist $\omega\in\real$ and $M>0$ such
that the resolvent $\varrho(A)=\{\lambda\in\mathbb{C}:\lambda-A \
{\rm is} \ {\rm invertible}\}$ contains the set
$S:=\{\lambda\in\mathbb{C}: {\rm Re}\lambda\geq\omega\}$ and
\begin{equation}\label{est.resol}
\|\lambda(\lambda-A)^{-1}x\|\leq M\|x\|
\end{equation}
holds true for all $\lambda\in S$ and $x\in X$, then $A$ is
sectorial from $D(A)$ to $X$.
\end{proposition}
Our strategy is to show that the linearized operator $L$ in
\eqref{linearized} is sectorial from $D(L):=\{v\in C^2([0,1]):
v'(0)=v(1)=0\}$ to $X=C([0,1])$. In order to do it, we take as
starting point the formal approximations in \eqref{approx.L} and the
fact that the operators $L_0$ and $L_{\e}$, defined in
\eqref{linearized.0} and \eqref{linearized.1} respectively, are both
sectorial. For $L_0$, this is proved by Angenent in \cite[Lemma
4.3]{Ang} in a more general context, see also \cite[Theorem
4.3]{GM}, while for $L_{\e}$, this follows from a more general
theory in \cite[Corollary 3.1.21.(ii) and Theorem 3.1.19]{Lunardi}.
We can combine these two facts into the following result:
\begin{proposition}
The operator $L$ defined in \eqref{linearized} is sectorial from
$D(L):=\{v\in C^2([0,1]): v'(0)=v(1)=0\}$ to $X=C([0,1])$. In
addition, there exists $C=C(M)>1$ depending on
$$
M:=(\inf a)^{-1}+(\inf
b)^{-1}+\|a\|_{\infty}+\|b\|_{\infty}+\|c\|_{\infty}+\|d\|_{\infty},
$$
where the infima and norms are taking over $r\in[0,1]$, such that,
for all $v\in D(L)$ and $\lambda\in\mathbb{C}$ such that ${\rm
Re}\lambda\geq C(M)$, we have
\begin{equation}\label{est.sector}
|\lambda|\|v\|_{\infty}+|\lambda|^{1/2}\|v'\|_{\infty}+\|v''\|_{\infty}\leq
C(M)\|(\lambda-L)v\|_{\infty}.
\end{equation}
\end{proposition}
\begin{proof}
The proof follows some ideas in \cite[Proposition 4.5]{GM} and uses
Proposition \ref{prop.sector}. We divide the proof into several
steps.

\noindent \textbf{Step 1.} We show that there exists some
$\omega_1>0$ sufficiently large such that $\lambda-L$ is injective
for all $\lambda\in\mathbb{C}$ with ${\rm Re}\lambda\geq\omega_1$.
Let $v\in D(L)$ be such that $(\lambda-L)v=0$. We want to show that
for $\omega_1$ sufficiently large, this implies $v\equiv0$. We begin
by the following calculation:
\begin{equation*}
\begin{split}
0&=\frac{1}{a}\left[r\overline{v}(\lambda-L)v+rv\overline{(\lambda-L)v}\right]\\
&=\frac{1}{a}\left[2r{\rm
Re}\lambda|v|^2-ar(\overline{v}v''+v\overline{v}'')-b(\overline{v}v'+v\overline{v}')-cr(\overline{v}v'+v\overline{v}')-2rd|v|^2\right]\\
&=-r(\overline{v}v''+v\overline{v}'')-\left(\frac{b}{a}+\frac{rc}{a}\right)(|v|^2)'-\frac{2r(d-{\rm
Re}\lambda)}{a}|v|^2.
\end{split}
\end{equation*}
Let $a_0=\inf a>0$, $b_0=\inf b>0$, $A_0=\|a\|_{\infty}$,
$B_0=\|b\|_{\infty}$, $C_0=\|c\|_{\infty}$, $D_0=\|d\|_{\infty}$. We
integrate over $(0,1)$ the previous equality and estimate it term by
term as follows.

\noindent (i) After integrating by parts and taking into account
that $v(1)=0$ and that $\lim\limits_{r\to0}(b(r)/a(r))=3$, we have:
\begin{equation*}
\begin{split}
-\int_0^1r(\overline{v}v''+v\overline{v}'')\,dr&-\int_0^1\frac{b}{a}(|v|^2)'\,dr=2\int_0^1r|v'|^2\,dr-\int_0^1\left(\frac{b}{a}-1\right)(|v|^2)'\,dr\\
&=2\int_0^1r|v'|^2\,dr-\left(\frac{b}{a}-1\right)|v|^2\Big|_{0}^1+\int_0^1\left(\frac{b}{a}\right)'|v|^2\,dr\\
&=2\int_0^1r|v'|^2\,dr+2|v(0)|^2+\int_0^1\left(\frac{b}{a}\right)'|v|^2\,dr.
\end{split}
\end{equation*}

\noindent (ii) We easily find that $$2{\rm
Re}\lambda\int_0^1\frac{r}{a}|v|^2 \,dr\geq\frac{2{\rm
Re}\lambda}{A_0}\int_0^1r|v|^2\,dr.$$

\noindent (iii) After using H\"older and Young's inequalities, we
get
\begin{equation*}
\begin{split}
\left|\int_0^1\frac{rc}{a}(|v|^2)'\,dr\right|&\leq\frac{2C_0}{a_0}\left(\int_0^1r|v|^2\,dr\right)^{1/2}\left(\int_0^1r|v'|^2\,dr\right)^{1/2}\\
&\leq\int_0^1r|v'|^2\,dr+\frac{C_0^2}{a_0^2}\int_0^1r|v|^2\,dr.
\end{split}
\end{equation*}

\noindent (iv) Obviously,
$$
\left|\int_0^1\frac{2rd}{a}|v|^2\,dr\right|\leq\frac{2D_0}{a_0}\int_0^1r|v|^2\,dr.
$$

Gathering the estimates (i)-(iv) above, we obtain
\begin{equation}\label{interm8}
2\int_0^1r|v'|^2\,dr+2|v(0)|^2+\left(\frac{2{\rm
Re}\lambda}{A_0}-\frac{C_0^2}{a_0^2}-\frac{2D_0}{a_0}\right)\int_0^1r|v|^2\,dr\leq-\int_0^1\left(\frac{b}{a}\right)'|v|^2\,dr.
\end{equation}
On the other hand,
\begin{equation*}
\begin{split}
\left(\frac{b}{a}\right)'(r)&=-\frac{ 6rpu_0}{\left[r^2(p -
1)u_{0,r}^2 + 2r(p - 1)u_0u_{0,r} + pu_0^2\right]^2(r^2u_0^2 + 1)^2}\times
\\ & \Big\{(p - 1)u_0u_{0,rr}(ru_{0,r} + u_0)(r^2u_0^2 + 1)+ r(1 - p)u_{0,r}^3 \\ & \left.+ u_0^2\left[3r^2(p - 1)u_0u_r^2 +
u_{0,r}(r(4p - 3)u_0^2)+
pu_0^3\right]+\frac{(p-1)u_0^2u_{0,r}}{r}\right\}
\end{split}
\end{equation*}
and taking into account that $u_0(0)>0$, $u_0'(0)=0$ and $u_0\in
C^2([0,1])$, we find that there exists some $\overline{C}>0$
sufficiently large such that
\begin{equation}\label{interm52}
\left|\left(\frac{b}{a}\right)'(r)\right|\leq\overline{C}r.
\end{equation}
Replacing this estimate into \eqref{interm8}, we obtain that
$$
2\int_0^1r|v'|^2\,dr+2|v(0)|^2+\left(\frac{2{\rm
Re}\lambda}{A_0}-\frac{C_0^2}{a_0^2}-\frac{2D_0}{a_0}-\overline{C}\right)\int_0^1r|v|^2\,dr\leq0,
$$
whence $v\equiv0$ if $\lambda\in\mathbb{C}$ is taken with ${\rm
Re}\lambda$ sufficiently large.

\medskip

\noindent \textbf{Step 2.} In order to prove the surjectivity of
$\lambda-L$, we characterize the inverse of $\lambda-L$ using a
gluing between approximate inverses, through a partition of unity.
Let $\e\in (0,1)$ to be chosen later. Since $a(0)=1$ and
$b(0)=3$, $a$, $b\in C([0,1])$, there exists $\delta=\delta(\e)>0$
such that
\begin{equation}\label{interm.est}
|a(r)-1|<\e, \quad |b(r)-3|<\e, \quad \hbox{for} \ \hbox{all} \
r\in(0,2\delta).
\end{equation}
Let $\varphi_0$, $\varphi_1\in C^{\infty}([0,1])$ such that
$\hbox{supp}\varphi_0\subseteq[0,2\delta]$,
$\hbox{supp}\varphi_1\subseteq[\delta,1]$ and
$\varphi_0^2(r)+\varphi_1^2(r)=1$ for any $r\in[0,1]$. Recall that
the operators $L_0$, $L_{\e}$ introduced in \eqref{linearized.0},
respectively \eqref{linearized.1}, are sectorial, hence there exists
$\omega(\e)\in[0,\infty)$ such that, for any $\lambda\in\mathbb{C}$
with ${\rm Re}\lambda\geq\omega(\e)$, $\lambda-L_0$ and
$\lambda-L_{\e}$ are both invertible. Thus, for any
$\lambda\in\mathbb{C}$ with ${\rm Re}\lambda\geq\omega(\e)$ and for
any $f\in C([0,1])$, we can define
$$
v_0:=(\lambda-L_0)^{-1}(\varphi_0f), \quad
v_1:=(\lambda-L_{\e})^{-1}(\varphi_1f),
$$
and the operators
\begin{equation}\label{oper}
T_{\lambda}:C([0,1])\mapsto D(L), \
T_{\lambda}f:=\varphi_0v_0+\varphi_1v_1, \quad
A_{\lambda}:C([0,1])\mapsto C([0,1]), \
A_{\lambda}f:=(\lambda-L)T_{\lambda}f.
\end{equation}
We state the following

\noindent \textbf{Claim.} There exists $\e>0$ sufficiently small and
$\omega_2>\omega(\e)$ sufficiently large such that
\begin{equation}\label{oper2}
\|A_{\lambda}f-f\|_{\infty}\leq\frac{1}{2}\|f\|_{\infty},
\end{equation}
for all $f\in C([0,1])$ and $\lambda\in\mathbb{C}$ with ${\rm
Re}\lambda\geq\omega_2$.

Assuming that the claim is true, then we can proceed exactly as in
the proof of \cite[Proposition 4.5]{GM} to find that first
$A_{\lambda}$ is invertible, then $\lambda-L$ is invertible too for
any $\lambda\in\mathbb{C}$ with ${\rm
Re}\lambda\geq\max\{\omega_1,\omega_2\}$, and the inverse is given
by $(\lambda-L)^{-1}=T_{\lambda}A_{\lambda}^{-1}$. It only remains
to prove the claim.

\medskip

\noindent \textbf{Step 3. Proof of the claim.} Let
$\lambda\in\mathbb{C}$ such that ${\rm
Re}\lambda\geq\max\{1,\omega(\e)\}$, with the notations in Step 2.
Since the operators $L_0$ and $L_{\e}$ are sectorial, there exist
two constants $C>0$ and $C(\e)>0$, the last one depending on $\e$,
such that
\begin{equation}\label{interm9}
|\lambda|\|v_0\|_{\infty}+|\lambda|^{1/2}\|v_0'\|_{\infty}+\|v_0''\|_{\infty}\leq
C\|f\|_{\infty}
\end{equation}
and
\begin{equation}\label{interm10}
|\lambda|\|v_1\|_{\infty}+|\lambda|^{1/2}\|v_1'\|_{\infty}+\|v_1''\|_{\infty}\leq
C(\e)\|f\|_{\infty}.
\end{equation}
By a straightforward calculation, we obtain that
\begin{equation}\label{interm11}
A_{\lambda}f-f=-[L,\varphi_0]v_0-[L,\varphi_1]v_1+\varphi_0(L_0-L)v_0+\varphi_1(L_{\e}-L)v_1,
\end{equation}
where
$$
[L,\varphi_0]v_0:=a\varphi_0''v_0+2a\varphi_0'v_0'+\frac{b}{r}\varphi_0'v_0+c\varphi_0'v_0,
\quad
[L,\varphi_1]v_1:=a\varphi_1''v_1+2a\varphi_1'v_1'+\frac{b}{r}\varphi_1'v_1+c\varphi_1'v_1.
$$
We estimate next the terms in \eqref{interm11}. Recalling that
$\varphi_{i}'=0$ in $[0,\delta]$, for $i=1,2$, and using
\eqref{interm9} and \eqref{interm10} we have
$$
\|[L,\varphi_0]v_0\|_{\infty}\leq
C_1\|v_0'\|_{\infty}+C_2\|v_0\|_{\infty}\leq
C|\lambda|^{-1/2}\|f\|_{\infty}
$$
and
$$
\|[L,\varphi_1]v_1\|_{\infty}\leq
C_3\|v_1'\|_{\infty}+C_4\|v_1\|_{\infty}\leq
C(\e)|\lambda|^{-1/2}\|f\|_{\infty}.
$$
On the other hand, taking into account \eqref{interm.est}, we also
have
\begin{equation*}
\begin{split}
\|\varphi_0(L_0-L)v_0\|_{\infty}&\leq\|\varphi_0(1-a)v_0''+\varphi_0\frac{3-b}{r}v_0'+\varphi_0cv_0'+\varphi_0dv_0\|_{\infty}\\
&\leq\e\|v_0''\|_{\infty}+\e\left\|\frac{v_0'}{r}\right\|_{\infty}+C\|v_0'\|_{\infty}+C\|v_0\|_{\infty}\leq
C(\e+|\lambda|^{-1/2})\|f\|_{\infty},
\end{split}
\end{equation*}
and similarly
$$
\|\varphi_1(L_{\e}-L)v_1\|_{\infty}\leq
C(\e)|\lambda|^{-1/2}\|f\|_{\infty}.
$$
Gathering all the previous estimates and replacing them into
\eqref{interm11}, we obtain that there exist $C_5>0$ and
$\tilde C=\tilde C(\e)>0$, depending only on the previous constants and $\e$,
such that
$$
\|A_{\lambda}f-f\|_{\infty}\leq\left[C_5|\lambda|^{-1/2}+\tilde C(\e)\right]\|f\|_{\infty}.
$$
Choosing then $\lambda\in\mathbb{C}$ such that
$C|\lambda|^{-1/2}<1/4$ and finally choosing $\e>0$ sufficiently
small such that $\tilde C(\e)<1/4$, we end the proof of the claim.

\medskip

\noindent \textbf{Step 4.} It only remains to prove the estimate
\eqref{est.sector}, that will imply immediately \eqref{est.resol}
and in particular that $L$ is sectorial, ending the proof. But this
is identical to the end of the proof of \cite[Proposition 4.5]{GM}
and we omit the details.
\end{proof}
In order to prove the existence of solutions to our nonlinear PDE,
we follow \cite{GM} and use results about sectorial operators
introduced by Lunardi, more precisely \cite[Theorem 4.6]{GM} or
\cite[Theorem 8.1.1 and Corollary 8.3.3]{Lunardi} to which we refer.
We recall the definition of the following spaces, for $\a\in(0,1)$
$$
C_{\a}^{\a}((0,T];C^2([0,1]))=\{v\in L^{\infty}((0,T);C^2([0,1])):
(t\mapsto t^{\a}v(t))\in C^{\a}((0,t];C^2([0,1]))\},
$$
and
$$
B_{\a}((0,T];C^{2\a}([0,1]))=\{g:(0,T]\mapsto C^{2\a}([0,1]):
\sup\limits_{t\in(0,T)}t^{\a}\|g(t)\|_{C^{2\a}}<\infty\}.
$$
With this notation, we state the following result:
\begin{proposition}\label{prop.exist}
Let $u_0\in C^2([0,1])$ such that $\inf u_0>0$, $u_{0r}(0)=0$ and
$G(u_0)|_{r=1}=0$, where $G(u)$ is defined in \eqref{eq.transf}.
Then there exists a time $T>0$ and a unique $u\in
C^{1,2}([0,T]\times[0,1])\cap C_{\a}^{\a}((0,T];C^2([0,1]))$ for all
$\a\in(0,1)$, which solve the problem
\eqref{eq.transf}-\eqref{eq.transfinit}. Moreover, $u$ depends
continuously on $u_0$ and $u_t\in B_{\a}((0,T];C^{2\a}([0,1]))$ for
any $\a\in(0,1)\setminus\{1/2\}$.
\end{proposition}
The proof is an immediate adaptation of that of \cite[Proposition
4.7]{GM} and we omit it. The previous existence result allows us to
prove Theorem \ref{th.exist}.
\begin{proof}[Proof of Theorem \ref{th.exist}] Let $h_0\in
C^2([0,1])$ satisfy conditions \eqref{initcond} and
\eqref{initcond2}. Then the function
$$
u_0(r):=\frac{1}{r}\tan\left(\frac{h_0(r)}{2}\right)
$$
is well defined and belongs to $C^2([0,1])$, and it is easy to check
that $\lim\limits_{r\to0^+}=h_0'(0)/2>0$, hence $\inf u_0>0$.
Moreover, from \eqref{initcond2} we deduce that $u_{0r}(0)=0$ and
that $G(u_0)|_{r=1}=F(h_0)|_{r=1}=0$. By Proposition
\ref{prop.exist}, there exist $T>0$ and a unique $u\in
C^{1,2}([0,T]\times[0,1])\cap C_{\a}^{\a}((0,T];C^2([0,1]))$ with
the properties described in Proposition \ref{prop.exist}. Thus,
undoing the change of variables, the function
$$
h(t,r)=2\arctan(ru(t,r))
$$
is the claimed solution. Then, noticing that
$$
h_r(t,r)=\frac{2(u(t,r)+ru_r(t,r))}{1+r^2u^2(t,r)}, \quad
(t,r)\in[0,T]\times[0,1],
$$
in order to prove part (b) in Theorem \ref{th.exist}, it is enough
to show that \begin{equation}\label{claim-new}u_r\in C^{1,2}([\e,T]\times[\e,1])\quad {\rm  for\  any\ }\e>0.\end{equation}
Let $\e>0$ and $\alpha\in (1/2,1)$. It follows from Proposition \ref{prop.exist} that
\begin{equation}
\label{add-1}
\sup_{t\in (\e,T)} \|u_t(t,\cdot)\|_{C^{2\alpha}([0,1])} <\infty.
\end{equation}
Since $u\in C^{1,2}([0,T]\times[0,1])$ and $\inf u>0$, we obtain
that the mapping
$$
(t,r)\mapsto(\overline{a}(t,r),\overline{f}(t,r))=(a(r,u(t,r),u_r(t,r)),f(r,u(t,r),u_r(t,r)),
$$
is of class $C([0,T];C^1([0,1]))$, while
$$
(t,r)\mapsto\overline{b}(t,r)=b(r,u(t,r))\in C([0,T];C^2([0,1])).
$$
Recalling that $\inf\overline{a}>0$, we define
$$
l(t,r):=\frac{1}{\overline{a}(t,r)}(u_t(t,r)-\overline{f}(t,r))=u_{rr}(t,r)+\frac{\overline{b}(t,r)}{\overline{a}(t,r)}\frac{u_r(t,r)}{r},
$$
then $l\in L^{\infty}([\e,T];C^1([0,1]))$.

In order to simplify the notations, in the following technical
estimates we omit the dependence on time variable. We first define
\begin{equation}\label{interm25}
w(r)=\left(\frac{u(0)}{u(r)}\right)^{6(p-1)/p}r^3\exp\left[\int_0^r
\left(\frac{\overline b}{s\overline
a}(s)-\frac{3}{s}+\frac{6(p-1)u_{r}}{pu}\right) d s\right]
\end{equation}and we note that
$$
\frac{w_r}{w}=\frac{\overline{b}}{\overline{a}r}
$$and $$w(r)=r^3(1+o(r)),$$
where the $o(r)$ symbol is taken as $r\to0^+$. By standard
integration by parts, we obtain
\begin{equation}\label{interm24}
u_r=\frac{1}{w(r)}\int_0^rw(\varrho)l(\varrho)\,d\varrho=\frac{1}{w(r)}\left[l(r)\int_0^rw(s)\,ds-\int_0^r\left(\int_0^{\varrho}w(s)\,ds\right)l_r(\varrho)\,d\varrho\right],
\end{equation}

We also have
$$
u_{rr}=l-\frac{\overline{b}}{\overline{a}r}u_r,
$$
whence
\begin{equation}\label{interm23}
\begin{split}
u_{rrr}&=l_r-\left(\frac{\overline{b}}{\overline{a}}\right)_r\frac{u_r}{r}+\frac{\overline{b}}{\overline{a}}
\left[\left(1+\frac{\overline{b}}{\overline{a}}\right)\frac{u_r}{r^2}-\frac{l}{r}\right]\\
&=l_r-\left(\frac{\overline{b}}{\overline{a}}\right)_r\frac{u_r}{r}-\frac{\overline{b}}{\overline{a}r^2}\left(1+\frac{\overline{b}}{\overline{a}}\right)\frac{1}{w}
\int_0^r\left(\int_0^{\varrho}w(s)\,ds\right)l_r(\varrho)\,d\varrho\\
&+\frac{\overline{b}l}{\overline{a}}\left[\left(1+\frac{\overline{b}}{\overline{a}}\right)\frac{1}{wr^2}\int_0^rw(s)\,ds-\frac{1}{r}\right]=L_1+L_2+L_3,
\end{split}
\end{equation}
where
$L_1=l_r-\left(\frac{\overline{b}}{\overline{a}}\right)_r\frac{u_r}{r}$.
Our next aim is to show that $u_{rrr}$ is bounded for $r\in[0,1]$.
Recall that $u_r/r$ is bounded in $[0,1]$ and notice that $u_{rrr}$
is uniformly bounded for $r\in[\delta,1]$ for any $\delta>0$, thus
we will only have to prove that the uniform boundedness can be
extended up to $r=0$. In order to estimate $L_1$, we take into
account \eqref{interm52} and the fact that $u_r(0)=0$, and we get
\begin{equation}\label{intermT1}
\lim\limits_{r\to0^+}I_1=l_r(0).
\end{equation}

On the other hand, since
$$
1+\frac{\overline{b}}{\overline{a}}=4-\frac{6(p-1)ru_r}{pu}+o(r)=4+o(r),
\quad {\rm as} \ r\to0^+,
$$
we can write
\begin{equation}\label{intermT3}
\begin{split}
L_3&=\frac{\overline{b}l}{\overline{a}}\left[\frac{4+o(r)}{r^5(1+o(r))}\left(\frac{r^4}{4}+\int_0^rs^3o(s)\,ds\right)-\frac{1}{r}\right]\\
&=\frac{\overline{b}l}{\overline{a}}\left[\frac{1}{r(1+o(r))}-\frac{1}{r}+\frac{4}{r^5(1+o(r))}\int_0^rs^3o(s)\,ds-o(1)\right]\to0
\quad {\rm as} \ r\to0^+.
\end{split}
\end{equation}
In a similar way, we can estimate
\begin{equation}\label{intermT2}
L_2=-\frac{\overline{b}}{\overline{a}}\left(1+\frac{\overline{b}}{\overline{a}}\right)\frac{1}{r^5(1+o(r))}\int_0^{r}\frac{s^4}{4}(1+o(s))l_r(s)\,ds
\to-\frac{3}{5}l_r(0),
\quad {\rm as} \ r\to0^+,
\end{equation}
Gathering \eqref{intermT3}, \eqref{intermT2} and \eqref{intermT1}
and taking limits as $r\to0^+$ in \eqref{interm23}, we obtain that
$u_{rrr}\in L^{\infty}([\e,T];C([0,1]))$, for any $\e>0$, and
moreover $u_{rrr}(t,0)=2l_r(t,0)/5$. This implies that  $\overline
a,\overline b$ and  $\overline  f$ belong to
$L^\infty([\e,T];C^2([0,1]))$. Then, recalling \eqref{add-1} and the
definition of $l$, we have $l\in L^\infty([\e,T];
C^{2\alpha}([0,1]))$, which  implies in particular that
\begin{equation}\label{add-2}
\sup_{t\in (\e,T)} \|u_{rr}(t,\cdot)\|_{C^{2\alpha}([\e,1])} <\infty.
\end{equation}
Interpolating between \eqref{add-1} and \eqref{add-2} (see e.g. \cite[Lemma 5.1.1]{Lunardi}) yields
$$
\sup_{r\in (\e,1)} \|u_{rr}(\cdot,r)\|_{C^{\alpha}([\e,T])} <\infty.
$$
Together with \eqref{add-2}, this means that $u_{rr} \in C^{\alpha,2\alpha}([\e,T]\times [\e,1])$, and using the equation we obtain that
\begin{equation}\label{add-6}
u\in C^{1+\alpha,2+2\alpha}([\e,T]\times[\e,1]).
\end{equation}

To conclude, we note that $u$ is the unique solution to the linear
problem
$$
(P_v) \left\{\begin{array}{ll}
v_t-\overline a v_{rr} = \frac{\overline b u_r}{r}+\overline f & \mbox{in }\ Q_{\e,T}:=[\e,T]\times [\e,1]
\\[1ex]
v=u & \mbox{on } \ \partial Q_{\e,T}\cap\{t<T\}.
\end{array}
\right.
$$
If we formally differentiate the equation once with respect to $r$,
we obtain the boundary value problem
$$
(P_w) \left\{\begin{array}{ll}
w_t-\overline a w_{rr}-\overline a_r w_r = \left(\frac{\overline b u_r}{r}+\overline f\right)_r=:\tilde F & \mbox{in }\ Q_{\e,T}
\\[1ex]
w_r=u_{rr} & \mbox{on } \ [\e,T] \times \partial [\e,1]
\\[1ex]
w(\e,r)=u_r(\e,r) & r\in [\e,1].
\end{array}
\right.
$$
Because of \eqref{add-6}, letting $\ell =2\alpha-1$ we have $\tilde
a,\, \tilde a_r,\, \tilde F \in
C^{(1+\ell)/2,1+\ell}(Q_{\e,T})\subset C^{\ell/2,\ell}(Q_{\e,T})$,
$u_{rr}|_{\partial[\e,1]}\in C^{(1+\ell)/2}([\e,T])$, and
$u_r(\e,\cdot)\in C^{2+\ell}([\e,1])$. Hence we may apply to $(P_w)$
the classical well-posedness result in H\"older spaces given in
\cite[Theorem IV.5.3]{LSU}, yielding the existence and uniqueness of
$w\in C^{1+\ell/2,2+\ell}(Q_{\e,T})$. Finally, it is not difficult
to check that $\tilde u(t,r)= u(t,\e)+\int_\e^r w(t,s)ds$ is a
solution to $(P_v)$: hence $\tilde u=u$ and $u_r =w\in
C^{1+\ell/2,2+\ell}(Q_{\e,T})\subset C^{1,2}(Q_{\e,T})$. Since
$\e\in (0,1)$ is arbitrary, \eqref{claim-new} holds and the proof is
complete.\qquad\end{proof}

\section{Bounds on the derivative}\label{sec.bound}

In this section, we establish some uniform interior bounds for
$h_r$, where $h$ is a generic solution to \eqref{pHarmRotSim}. This
will be done through a maximum and minimum principle applied to some
suitable PDEs.
\begin{proposition}\label{prop.est}
Let $h_0$ satisfy \eqref{initcond} and \eqref{initcond2}. Then,
there exists $K=K(h_0(1))$ such that the classical solution $h$ to
\eqref{pHarmDirichlet} with initial condition $h_0$, given by
Theorems \ref{th.unique} and \ref{th.exist}, satisfies
\begin{equation}\label{derivbound}
\sup\limits_{(t,r)\in[0,T)\times(0,1)}|rh_r(t,r)|\leq\max\left\{\sup\limits_{r\in[0,1]}\left(\frac{\pi}{h_0(r)}\right)^{2/3}|rh_{0r}(r)|,K\right\}.
\end{equation}
\end{proposition}
Notice that the right hand side of \eqref{derivbound} is finite due
to assumptions in \eqref{initcond}.
\begin{proof} We prove the lower and the upper bound in two different steps.

\medskip

\noindent \textbf{Step 1. The upper bound.} We first take
$$d(t,r):=\frac{(rh_r(t,r))^{1/\gamma}}{h(t,r)^{(1-\gamma)/\gamma}},$$
with $\gamma>0$ to be chosen later.  Taking into account that
\begin{equation}\label{interm18}
d_r=\frac{d}{r\gamma}+\frac{\gamma-1}{\gamma}\frac{d^{\gamma+1}}{rh^{\gamma}}+\frac{rd^{1-\gamma}}{\gamma
h^{1-\gamma}}h_{rr},
\end{equation}after some straightforward computations, we then notice that the partial
differential equation satisfied by $d$ has the form:
\begin{equation}\label{estmax1}
z(r,h,d)d_t=a(r,h,d)d_{rr}+b(r,h,d,d_r)d_r+c(h,d),
\end{equation}
where $$z(r,h,d)=r^p((d^\gamma
h^{1-\gamma})^2+\sin^2h)^{\frac{6-p}{2}}>0$$
$$a(r,h,d)=r^2\sin^2h(p(d^\gamma h^{1-\gamma})^4+\sin^2h)>0,$$
$b(r,h,d,d_r)\in C([0,1]\times[0,\pi]\times\real^2)$ and
$$
c(h,d)=\sum\limits_{j=0}^7c_jd^{1+(j-1)\gamma},
$$
where $c_j$ are polynomials in $(h,\sin h,\cos h)$; we omit the
details of the calculation. Notice that $c$ does not depend on $r$,
due to the homogeneity of Eq. \eqref{pHarmRotSim}. We thus only want
to evaluate the sign of $c_7$, that is, the dominating term with
respect to the powers of $d$, in order to verify that a maximum
principle holds true. We notice that
$$c_7(h)=\frac{\gamma-1}{\gamma}(p-1)h^{4-6\gamma}((p-1)\gamma+2-p)$$
If we choose now $\gamma=1/3$ so that $4-6\gamma=2$, then $c_7(h)<0$, which implies that
there exists a universal constant $d_0>0$ such that $c(h,d)<0$ for
all $d>d_0$ and $h\in[0,\pi]$. Let us define
$$
K_0:=\max\left\{d_0,\sup\limits_{r\in[0,1]}\frac{(rh_{0r}(r))^{3}}{h_0^{2}(r)},K_1\right\},
\quad K_1:=\left(\frac{3(3-p)}{4(p-1)h_0^{1/3}(1)}\right)^{3/2}.
$$
Assume by contradiction that there exists a point $(t_0',r_0')$ such
that $d(t_0',r_0')>K_0$, and let $(t_0,r_0)$ be the point where $d$
attains its maximum in the compact set $[0,t_0']\times[0,1]$. Then
$d(t_0,r_0)>K_0$ and obviously $t_0>0$, $r_0>0$. We also want to
avoid $r=1$. At $r=1$, we have $h_t=0$ from the Dirichlet condition,
hence (everything being evaluated at $(t,1)$)
\begin{equation}\label{interm21}
h_{rr}=\frac{\left[h_r^2(3-p)+\sin^2h\right]\sin h\cos
h-\left[h_r^2+(3-p)\sin^2h\right]h_r}{(p-1)h_r^2+\sin^2h}.
\end{equation}
We also notice that, for $p\in(1,2)$, we have
\begin{equation}\label{interm22}
1<\frac{(3-p)h_r^2+\sin^2h}{(p-1)h_r^2+\sin^2h}<\frac{3-p}{p-1},
\quad
1<\frac{h_r^2+(3-p)\sin^2h}{(p-1)h_r^2+\sin^2h}<\frac{3-p}{p-1},
\end{equation}
whence
\begin{equation}\label{interm20}
h_{rr}(t,1)<\frac{(3-p)}{2(p-1)}-h_r(t,1)=\frac{(3-p)}{2(p-1)}-d(t,1)^{1/3}h_0(1)^{2/3}.
\end{equation}
Then, replacing \eqref{interm20} into \eqref{interm18}, we obtain
$$
d_r(t,1)<\frac{3d(t,1)^{2/3}}{
h_0(1)^{2/3}}\left[\frac{3-p}{2(p-1)}-\frac{2}{3}d(t,1)^{2/3}h_0(1)^{1/3}\right],
$$
whence $d_r(t,1)<0$ if $d(t,1)>K_0>K_1$. It follows that $r_0<1$,
hence $d_t(t_0,r_0)\geq0$, $d_r(t_0,r_0)=0$ and
$d_{rr}(t_0,r_0)\leq0$. Evaluating \eqref{estmax1} at $(t_0,r_0)$,
we find  $c(h(t_0,r_0),d(t_0,r_0))$ $\geq0$, which contradicts the
choice of $d(t_0,r_0)>K_0>d_0$. Thus, $d(t,r)\leq K_0$ in
$[0,T)\times[0,1]$, hence
$$
rh_r(t,r)=d(t,r)^{1/3}h(t,r)^{2/3}\leq
K_0^{1/3}\pi^{2/3}=:K,
$$
which is the claimed upper bound.

\medskip

\noindent \textbf{Step 2. The lower bound.} For the lower bound,
things are technically simpler. We let
$$d(t,r)=rh_r(t,r)(\pi+h(t,r))$$ and we follow the same ideas as in
Step 1. In this case,
the equation satisfied by $d$ has the same form as in
\eqref{estmax1}, where  $$z(r,h,d)=r^p(\pi+h)^{p+2}((\pi+h)^2\sin^2 h+pd^2)>0,$$ $$a(r,h,d)=r^2(\pi+h)^6\sin^2h((\pi+h)^2\sin^2h+pd^2)>0,$$
$b(r,h,d,d_r)\in C([0,1]\times[0,\pi]\times\real^2)$ and
$$
c(h,d)=\sum\limits_{j=0}^7c_jd^{j},
$$
where $c_j$ are polynomials in $(h,\sin h,\cos h)$.  In this case, $$c_7=p(p-1)>0.$$ 
Thus there exists a universal constant $d_1>0$ such that $c(h,d)>0$
for all $d\in(-\infty,-d_1]$ and $h\in[0,\pi]$. Let us define, as in
Step 1,
$$
\overline{K}_0:=\max\left\{d_1,\sup\limits_{r\in(0,1)}|r(\pi+h_0(r))h_{0r}(r)|,\overline{K}_1\right\},
$$$$
\overline{K}_1:=\frac{2\pi^2}{p-1}\left[\sqrt{\frac{4(p-2)^2\pi+(p-1)^2}{\pi}}-2(2-p)\right].
$$
Assume by contradiction that there exists a point
$(t_0',r_0')\in[0,T)\times[0,1]$ such that
$d(t_0',r_0')<-\overline{K}_0$, and let $(t_0,r_0)$ be the point
where $d(t,r)$ attains its minimum in the compact set
$[0,t_0']\times[0,1]$; obviously, $d(t_0,r_0)<-\overline{K}_0$ and,
by the conditions on $h_0$, we have $t_0>0$, $r_0>0$. We want to
avoid the case $r=1$; thus, evaluating \eqref{pHarmRotSim} at $r=1$,
we obtain \eqref{interm21} and we use again the estimates
\eqref{interm22}, arriving to the following inequality:
$$
h_{rr}(t,1)>-\frac{1}{2}-\frac{(3-p)d(t,1)}{(p-1)(\pi+h_0(1))},
$$
hence (everything being taken at $(t,1)$, that we omit from the
writing)
\begin{equation*}
\begin{split}
d_r&=(\pi+h)h_r+h_r^2+(\pi+h)h_{rr}\geq
d+\frac{d^2}{(\pi+h)^2}-\frac{1}{2}(\pi+h)-\frac{3-p}{p-1}d\\&=\frac{d^2}{(\pi+h)^2}-\frac{2(2-p)}{p-1}d-\frac{1}{2}(\pi+h)
\geq\frac{d^2}{4\pi^2}-\frac{2(2-p)}{p-1}d-\pi>0,
\end{split}
\end{equation*}
when $d(t,1)<-\overline{K}_0$, which implies that $r_0<1$. Thus,
$d_t(t_0,r_0)\leq0$, $d_r(t_0,r_0)=0$ and $d_{rr}(t_0,r_0)\geq0$.
Evaluating \eqref{estmax1} at $(t_0,r_0)$ and taking into account
the previous inequalities, we obtain that
$c(h(t_0,r_0),d(t_0,r_0))\leq0$, which contradicts the assumption
that $d(t_0,r_0)<-\overline{K}_0<-d_1$. Therefore
$d(t,r)>-\overline{K}_0$ in $[0,T)\times[0,1]$, yielding
$$
rh_r(t,r)>-\frac{\overline{K}_0}{\pi+h}>-\frac{\overline{K}_0}{\pi},
$$
which is the desired lower bound.
\end{proof}

\section{Blow up in finite time}\label{sec.blowup}

As described in the Introduction, in this section we prove Theorem
\ref{th.blowup1}. It relies on three steps marked as (A), (B) and
(C) in the Introduction, among which it remains to prove only part
(A), that is made precise in the following statement
\begin{proposition}\label{prop.asympt}
Let $h_0$, $T$ and $h$ be as in Theorem \ref{th.exist}. If $h$ is a
global solution (that is, $T=\infty$), then for any sequence
$t_n\to\infty$, there exists a subsequence (not relabeled) and a
stationary state $\overline{g}\in C^{\infty}((0,1])$ of
\eqref{pHarmRotSim} such that
$$
\lim\limits_{n\to\infty}\|h(t_n)-\overline{g}\|_{L^{\infty}([\e,1])}=0,
\quad {\rm for \ any} \ \e>0
$$
.
\end{proposition}
The proof is very similar to that of \cite[Lemma 7.1]{GM}, thus we
only sketch it. Before starting, let us define the following
\begin{equation}\label{interm26}
B(h):=\left((p-1)h_r^2+\frac{\sin^2h}{r^2}\right)^{-1}A(h)=h_{rr}+R(h)\frac{h_r}{r}-S(h)\frac{\sin
h\cos h}{r^2},
\end{equation}
with
$$
R(h):=\frac{(rh_r)^2+(3-p)\sin^2h}{(p-1)(rh_r)^2+\sin^2h}, \quad
S(h):=\frac{(3-p)(rh_r)^2+\sin^2h}{(p-1)(rh_r)^2+\sin^2h}.
$$
\begin{proof}[Proof of Proposition \ref{prop.asympt}]
Assume that $T=\infty$ and take any sequence $t_n\to\infty$ and any
$\e\in(0,1)$. In this proof, by $C_{\e}$ we will always understand a
generic positive constant depending on $\e$, which may change
without relabeling. From Proposition \ref{prop.est}, there exists a
constant $C_{\e}>0$, depending on $\e$, such that
\begin{equation}\label{interm27}
|h_r|\leq C_{\e}, \quad \frac{1}{C_{\e}}\leq\left|\frac{\sin
h}{r}\right|\leq C_{\e}, \quad J(h)\leq C_{\e},
\end{equation}
for any $(t,r)\in[0,\infty)\times[\e,1]$. Using then the
Arzelà-Ascoli theorem, there exists a subsequence (not relabeled)
such that $h(t_n)\to\overline{h}$ locally uniformly in $(0,1]$, as
$n\to\infty$, the argument being made completely rigorous in the
proof of \cite[Lemma 7.1]{GM}. Following as there, we let
$$
g_n(s,r):=h(t_n+s,r), \quad (s,r)\in[0,\infty)\times[0,1],
$$
whence $g_n$ satisfies the following bounds uniformly in $n$, which
are a consequence of \eqref{interm27}:
\begin{equation}\label{interm28}
|g_{nr}|\leq C_{\e}, \quad \frac{1}{C_{\e}}\leq\left|\frac{\sin
g_n}{r}\right|\leq C_{\e}, \quad J(g_n)\leq C_{\e},
\end{equation}
for any $(s,r)\in [0,\infty)\times[\e,1]$. Then
$$
B(g_n)=\left((p-1)h_r^2+\frac{\sin^2h}{r^2}\right)^{-1}J(g_n)^{(4-p)/2}g_{ns}\leq\frac{r^2}{\sin^2g_n}C_{\e}g_{ns},
$$
whence
$$
B(g_n)^2\leq C_{\e}rg_{ns}^2, \quad (s,r)\in[0,\infty)\times[\e,1].
$$
On the other hand, it is straightforward to check that
$$
1\leq R(h),S(h)\leq\frac{3-p}{p-1},
$$
for any $r\in[0,1]$, $h$ and $p\in(1,2)$, thus, from the definition
of $B$ and the bounds \eqref{interm28}, we get
\begin{equation*}
\begin{split}
g_{nrr}&=B(g_n)-R(g_n)\frac{g_{nr}}{r}+S(g_n)\frac{\sin g_n\cos
g_n}{r^2}\\&\leq
B(g_n)-\frac{g_{nr}}{r}+\frac{3-p}{p-1}\frac{C_{\e}}{r}\leq
C_{\e}(1+|B(g_n)|),
\end{split}
\end{equation*}
by modifying conveniently the constant $C_{\e}$ in each step. We
thus have the same bounds as in the proof of \cite[Lemma 7.1]{GM},
allowing us to continue as there and deduce that there exists a
limit $\overline{g}$ as $n\to\infty$ and it coincides with our
previous $\overline{h}$, giving the desired locally uniform
convergence.

It only remains to show that $\overline{g}$ is a classical
stationary solution to \eqref{pHarmRotSim}, that is,
$B(\overline{g})=0$. Let $\e>0$ and some test function $\varphi\in
C_{0}^{\infty}((\e,1])$. We know that $B(g_n)\to0$ in
$L^2([0,1];L^2_{{\rm loc}}((0,1]))$, thus
$$
0=\lim\limits_{n\to\infty}\int_0^1\int_0^1B(g_n)\varphi drdt=\lim\limits_{n\to\infty}\int_0^1\int_0^1\left(g_{nrr}+\frac{R(g_n)g_{nr}}{r}-\frac{S(g_n)\sin
g_n\cos g_n}{r^2}\right)\varphi drdt.
$$
Following the same argument as in part (II) of the proof of
\cite[Lemma 7.1]{GM} which involves the Aubin-Simon compactness
criterium, we obtain that, for a subsequence (not relabeled),
$g_n\to\overline{g}$ in $L^2([0,1];H^1([\e,1]))$ and in particular
$g_{nr}\to\overline{g}_r$ a.\,e. Thus, we can easily pass to the
limit in the right-hand side above to obtain that
$B(\overline{g})=0$ and that $\overline{g}$ has the required
smoothness. We omit the details.
\end{proof}
Joining Proposition \ref{prop.asympt} and Corollary \ref{cor}, we
have the following immediate consequence whose proof we skip.
\begin{corollary}\label{cor2}
Let $h_0$, $h$, $T$ be as in Theorem \ref{th.exist} and such that
$h_0(1)=l\in(H,\pi)$ and $T=\infty$. Then we have
$$
\lim\limits_{t\to\infty}\|h(t)-H_{l}(t)\|_{L^{\infty}([\e,1])}=0,
$$
for all $\e>0$, where $H_l$ is the stationary solution to
\eqref{pHarmRotSim} introduced in Corollary \ref{cor}.
\end{corollary}
We finally have all the needed ingredients for the proof of Theorem
\ref{th.blowup1}
\begin{proof}[Proof of Theorem \ref{th.blowup1}]
Assume by contradiction that $h$ is a global solution such that
$h(0)=0$, $h(1)=l\in(H,\pi)$. Then, Corollary \ref{cor2} holds true
for $h$. Since $H_l(0)=\pi>(p+4)\pi/6$, we can take some $\theta>0$
such that $H_l(\theta)>(p+4)\pi/6$. By Corollary \ref{cor2}, there
exists $t_0>0$ such that
$$
h(t,\theta)>\frac{(p+4)\pi}{6}, \quad {\rm for} \ {\rm all} \ t\geq
t_0.
$$
Define now the rescaled function $\tilde{h}(t,r):=h(t_0+\theta
t,\theta r)$. Then $\tilde{h}$ is a solution to \eqref{pHarmRotSim}
in $[0,\infty)\times[0,1]$, such that $\tilde{h}(0,r)=h(t_0,\theta
r)$ for $r\in[0,1]$, and the Dirichlet conditions
$$\tilde{h}(t,0)=0, \quad
\tilde{h}(t,1)>\frac{(p+4)\pi}{6}, \quad {\rm for} \ t>0.$$ Consider
now $f$ a subsolution from the family defined in Proposition
\ref{prop.subs}, choosing $b_0$ so large such that
$$
f(0,r)=\frac{p+4}{3}\arctan\left(\frac{r}{b_0}\right)\leq\tilde{h}(0,r)=h(t_0,\theta
r).
$$
Since $f(t,1)\leq(p+4)\pi/6\leq\tilde{h}(t,1)$, for any $t>0$, by
the comparison principle in Section \ref{sec2}, we obtain
$f\leq\tilde{h}$ in $[0,b_0/\delta_0)\times[0,1]$. But this is a
contradiction with the regularity of $\tilde{h}$ when taking limits
as $t\to b_0/\delta_0$, hence $T<\infty$.
\end{proof}
We end with the proof of the nongeneric blow up in finite time.
\begin{proof}[Proof of Theorem \ref{th.blowup2}]
We only sketch the proof, as it follows the lines of \cite[Theorem
4]{GM}. Let $g_0$ be as in Theorem \ref{th.exist} with
$g_0(1)=l\in(0,\pi)$ and $g_0(1/2)>(p+4)\pi/6$, and let $g$ be the
unique solution to the Dirichlet problem \eqref{pHarmDirichlet} with
initial datum $g_0$. Then by continuity, there exists a time
$t_0'>0$ such that $g(t,1/2)>(p+4)\pi/6$, for $t\in[0,t_0']$. Let
now $f$ be a subsolution from the family defined in Proposition
\ref{prop.subs}, with $b_0>0$ such that $b_0/\delta_0<t_0'$. Thus,
the desired initial datum is a function $h_0$ satisfying the
regularity hypothesis in Theorem \ref{th.exist}, $h_0(1)=l$ and
$$
h_0(r)\geq\left\{\begin{array}{ll} \max\{f(0,r),g_0(r)\}, \quad {\rm
if} \ r\in(0,1/2),\\g_0(r), \quad {\rm if} \
r\in(1/2,1).\end{array}\right.
$$
Let $h$ be the solution with initial datum $h_0$ and $T$ its maximal
existence time; let also $t_0=\min\{b_0/\delta_0,T\}<t_0'$. By
standard comparison, $h\geq g$, thus in particular $h(t,1/2)\geq
g(t,1/2)\geq(p+4)\pi/6>f(t,1/2)$ for all $t\in[0,t_0]$. We can thus
apply once more the comparison principle to compare $h$ and $f$ to
get that $f\leq h$ in the parabolic cylinder $[0,t_0]\times[0,1/2]$.
Since $h(t,0)=0$ for all $t<T$ and
$$
\lim\limits_{t\to(b_0/\delta_0)}h(t,r)\geq\lim\limits_{t\to(b_0/\delta_0)}f(t,r)=\frac{(p+4)\pi}{6},
\quad {\rm for} \ {\rm all} \ r\in(0,1],
$$
it follows that $T<b_0/\delta_0<\infty$, as claimed.
\end{proof}

\section{Appendix}

This section is devoted to prove Proposition \ref{prop.subs}. We
first note that the statement is somehow analogous to \cite[Lemma
5.1]{GM}, although by letting $p\to1$ we do not recover the same
subsolutions as there, but their form is similar. However, the proof
in our case is much more involved and highly technical, due to the
complicated form of \eqref{pHarmRotSim}. We will first state and
prove an inequality that will be useful at several points in the
proof.
\begin{lemma}\label{lem.interm}
For $k\in(2,4)$, the following inequality holds
true:
\begin{equation}\label{ineq.interm}
\sin\left(k\arctan\left(\frac{1}{x}\right)\right)<\frac{kx}{1+x^2}
\end{equation}
\end{lemma}
\begin{proof}
Let $\theta:=k\arctan(1/x)$. We first note that if $\theta\in
[\pi,2\pi]$, then the inequality is trivial. Let now
$$
g(k,x):=\frac{kx}{1+x^2}-\sin\left(k\arctan\left(\frac{1}{x}\right)\right).
$$
We differentiate it with
respect to $k$
$$
\frac{\partial g}{\partial
k}(k,x)=\frac{x}{1+x^2}-\cos\left(k\arctan\left(\frac{1}{x}\right)\right)\arctan\left(\frac{1}{x}\right),
$$
and
$$
\frac{\partial^2g}{\partial
k^2}(k,x)=\sin\left(k\arctan\left(\frac{1}{x}\right)\right)\arctan^2\left(\frac{1}{x}\right).
$$
Since  $\theta\in(0,\pi)$, this second derivative is strictly
positive, whence the first derivative increases with $k$. It follows
that
$$
\frac{\partial g}{\partial k}(k,x)\geq\frac{\partial g}{\partial
k}(2,x)=\frac{x}{1+x^2}-\frac{x^2-1}{x^2+1}\arctan\left(\frac{1}{x}\right)>0,
$$
the last inequality being trivial for $x\leq1$ and following from
the fact that
$$
\arctan\left(\frac{1}{x}\right)<\frac{1}{x}=\frac{x}{x^2}<\frac{x}{x^2-1},
$$
when $x>1$. Thus, $g$ increases with $k$ for all
$x\geq 0$, hence $g(k,x)\geq g(2,x)=0$.
\end{proof}
\begin{proof}[Proof of Proposition \ref{prop.subs}]
By differentiating it with respect to time and letting $s:=b/r$,
\eqref{spec.subs} writes
\begin{equation}\label{interm12}
\frac{p+4}{3}\delta\leq
r(1+s^2)J(h)^{(p-4)/2}A(h)=\frac{J(h)^{(p-4)/2}}{r^3(1+s^2)^{3/2}}\left(r^4(1+s^2)^{5/2}A(h)\right)
\end{equation}
Thinking at $s\in(0,\infty)$ as an independent variable (not
depending on time), it is enough to prove that the right-hand side
of \eqref{interm12} is uniformly positive for any $p\in(1,2)$. To
this end, we first notice, by straightforward calculations, that, as
$s\to\infty$
$$
\frac{J(h)^{(p-4)/2}}{r^3(1+s^2)^{3/2}}\sim2^{(p-4)/2}\left(\frac{p+4}{3}\right)^{p-4}\frac{1}{(rs)^{p-1}}=C_1(p)b^{1-p}>0,
$$
since $b<b_0<\infty$ and $p>1$. On the other hand, letting
$$
A^*(h):=r^4(1+s^2)^{5/2}A(h), \quad
\theta=\theta(s,p):=\frac{p+4}{3}\arctan\left(\frac{1}{s}\right),
$$
and considering $(s,p)$ as independent variables from now on, we
have
\begin{equation*}
\begin{split}
A^*(s,p)&=-\frac{3(1+s^2)^2\sin(2\theta)\left[9(1+s^2)^2\sin^2\theta+(3-p)(p+4)^2s^2\right]}{54(1+s^2)^{3/2}}\\
&-\frac{2s(p+4)\left[9(1+s^2)^2((p-3)s^2+p-1)\sin^2\theta+(p+4)^2s^2(2p-3-s^2)\right]}{54(1+s^2)^{3/2}}\\
&=-\frac{1}{54(1+s^2)^{3/2}}(I_1+I_2+I_3+I_4),
\end{split}
\end{equation*}
where
$$
I_1=27(1+s^2)^4\sin(2\theta)\sin^2\theta, \quad
I_2=3(3-p)(p+4)^2s^2(1+s^2)^2\sin(2\theta),
$$
and
$$
I_3=18(p+4)s(1+s^2)^2\left[(p-3)s^2+p-1\right]\sin^2\theta, \quad
I_4=2(p+4)^3s^3(2p-3-s^2).
$$
By a rather long but easy calculation that takes into account that
$\sin\theta\sim(p+4)/3s$, as $s\to\infty$, we find
$$
A^*(s,p)\to-\frac{(p+4)(p^5+10p^4+72p^3+256p^2-128p-1536)}{729}=C_2(p)>0
\quad \hbox{as} \ s\to\infty,
$$
which holds true for $p<2$. Thus, in order to get the desired
uniform positivity, it only remains to prove that
$I_1+I_2+I_3+I_4<0$ for any $s\in(0,\infty)$. This will be done in a
series of technical lemmas.

\begin{lemma}\label{lem.T1T3}
With the notation and conditions above, $I_1+I_3<0$ for $p\in(1,2)$,
$s\in(0,\infty)$.
\end{lemma}
\begin{proof}
It suffices to show that
$$
I_5=I_5(s,p):=3(1+s^2)^2\sin(2\theta)+2s(p+4)\left[(p-3)s^2+p-1\right]\leq0.
$$
This will be done by showing that $I_5$ is strictly increasing with
$p\in(1,2]$ and $I_5(s,2)=0$ for any $s>0$. Differentiating it once
with respect to $p$, we get
$$
\frac{\partial I_5}{\partial
p}=6(1+s^2)^2\cos(2\theta)\theta_p+2s\left[(2p+1)s^2+2p+3\right].
$$
If $\theta\in[0,\pi/4]\cup[3\pi/4,\pi]$, then obviously
$\frac{\partial I_5}{\partial p}>0$. If $\theta\in(\pi/4,3\pi/4)$,
we differentiate two times more with respect to $p$:
$$
\frac{\partial^2I_5}{\partial
p^2}=4(1+s^2)\left[s-\frac{1}{3}(1+s^2)\arctan^2\left(\frac{1}{s}\right)\sin(2\theta)\right],
$$
and
$$
\frac{\partial^3I_5}{\partial
p^3}=-\frac{4}{9}(1+s^2)^2\arctan^3\left(\frac{1}{s}\right)\cos(2\theta)>0,
$$
taking into account the range of $\theta$. It follows that
$\frac{\partial^2I_5}{\partial p^2}$ is increasing with $p$ and in
particular attains its minimum either for $p\to1$ or for the value
of $p$ for which $\theta=\pi/4$, in case this value of $p$ is
greater than 1. In the latter case, we can restrict the range of $s$
by noticing that $p=3\pi/(4\arctan(1/s))-4>1$, and at the same time
$\theta=(p+4)\arctan(1/s)/3=\pi/4$, whence this second case holds
only if
$$
\frac{1}{\tan(3\pi/20)}<s<\frac{1}{\tan{\pi/8}}.
$$
In this range we have
\begin{equation}\label{interm13}
\frac{\partial^2I_5}{\partial
p^2}\geq4(1+s^2)\left[s-\frac{1+s^2}{3}\arctan^2\left(\frac{1}{s}\right)\right].
\end{equation}
But it is easy to check that in the indicated range for $s$ (which
in particular implies $s>1$), we have
$$
\arctan^2\left(\frac{1}{s}\right)\leq\frac{1}{s^2}<\frac{3s}{1+s^2},
$$
whence the right-hand side in \eqref{interm13} is positive. In the
former case (when the minimum is attained as $p\to1$), we get
\begin{equation}\label{interm14}
\frac{\partial^2I_5}{\partial
p^2}\geq4(1+s^2)\left[s-\frac{1+s^2}{3}\arctan^2\left(\frac{1}{s}\right)\sin\left(\frac{10}{3}\arctan\left(\frac{1}{s}\right)\right)\right].
\end{equation}
To prove that the right-hand side in \eqref{interm14} is positive,
we first notice that if $\theta\in[\pi/2,3\pi/4)$ we are done, since
$\sin(2\theta)\leq0$. It remains to study the case
$\theta\in[\pi/4,\pi/2)$, which implies again a restriction for the
range of $s$ to
$$
\frac{1}{\tan(3\pi/10)}<s<\frac{1}{\tan(3\pi/20)}.
$$
We use then Lemma \ref{lem.interm} for $k=10/3$, obtaining that
\begin{equation}\label{interm15}
\sin\left(\frac{10}{3}\arctan\left(\frac{1}{s}\right)\right)<\frac{10s}{3(1+s^2)},
\end{equation}
and, due to the previous restriction on $s$
 we also get that
$\arctan^2(1/s)<9\pi^2/100$. Joining this with \eqref{interm15}, we obtain
that the right-hand side of \eqref{interm14} is positive in our
range, which shows furthermore that $\frac{\partial I_5}{\partial
p}$ increases with $p$.

In particular, $\frac{\partial I_5}{\partial p}$ attains its minimum
either for $p\to1$ or for the value of $p$ for which $\theta=\pi/4$,
in case this value of $p$ is greater than 1. In the latter, we
trivially get
$$
\frac{\partial I_5}{\partial p}\geq2s\left[(2p+1)s^2+2p+3\right]>0.
$$
In the former, we replace $p=1$ to obtain
$$
\frac{\partial I_5}{\partial
p}\geq2\left[(1+s^2)^2\cos\left(\frac{10}{3}\arctan\left(\frac{1}{s}\right)\right)\arctan\left(\frac{1}{s}\right)+s(3s^2+5)\right]
$$
In order to prove that the right-hand side above is positive for
$\theta\in(\pi/4,3\pi/4)$, we restrict the range of $s$ to
$s\in(1/\tan(9\pi/20),1/\tan(3\pi/20))$. Since $10/3\arctan(1/s)\in [\frac{\pi}{2},\frac{5\pi}{3}]$, then
$
\cos\left(\frac{10}{3}\arctan\left(\frac{1}{s}\right)\right)>-\frac{1}{2}.
$
Hence, we only have to check that
$$
l(s):=s(3s^2+5)-\frac{1}{2}(1+s^2)^2\arctan\left(\frac{1}{s}\right)>0,
\quad \hbox{for} \ \hbox{any} \
s\in\left(\frac{1}{\tan(9\pi/20)},\frac{1}{\tan(3\pi/20)}\right).
$$
It is easily seen that $l$ is increasing in $s$ and there exists a
unique $s_1>0$ such that $l(s_1)=0$. By approximation using for
example Newton's method, we find that $s_1<0.15<1/\tan(9\pi/20)$,
which shows that we have covered all the admissible interval for
$s$.

Thus, $I_5$ is increasing with respect to $p$ and it follows that
$I_5<0$ for $p\in(1,2)$, since $I_5|_{p=2}\equiv0$.
\end{proof}
\begin{lemma}\label{lem.T2T3}
Under the notation and conditions above, $I_2+I_3<0$ for any
$p\in(1,2)$ and $s>\sqrt{3/7}$.
\end{lemma}
\begin{proof}
The inequality to prove is equivalent to
$$
(3-p)(p+4)s\cos\theta+3\left[(p-3)s^2+p-1\right]\sin\theta<0,
$$
or furthermore
$$
\varphi(s):=\theta(s,p)-\hbox{arccot}\left(-\frac{3(p-3)s^2+3(p-1)}{(3-p)(p+4)s}\right)>0.
$$
We differentiate once with respect to $s$ and find
$$
\varphi'(s)=\frac{(2-p)(p+4)\left[s^2(p-3)(p^2+7p+24)+18(p-1)\right]}{3(1+s^2)\left[9s^4(p-3)^2+s^2(p-3)(p^3+5p^2+10p-66)+9(p-1)^2\right]}=\frac{A(s,p)}{B(s,p)}.
$$
It is obvious that $B(s,p)>0$. On the other hand, as it easy to
check, for $1<p<2$,
$$
\frac{18(p-1)}{(3-p)(p^2+7p+24)}\leq\frac{3}{7}<s^2,
$$
whence $A(s,p)<0$ in our range for $s$ and $p$. Thus, $\varphi(s)$
is decreasing with $s$ for $s>\sqrt{3/7}$, hence
$\varphi(s)>\lim\limits_{s\to\infty}\varphi(s)=0$, which ends the
proof.
\end{proof}
\begin{lemma}\label{lem.T1T4}
Under the notation and conditions above, $I_1+I_4<0$ for any
$p\in(1,2)$.
\end{lemma}
\begin{proof}
Let $\theta=\theta(s,p)=\frac{p+4}{3}\arctan(\frac{1}{s})\in
[0,\pi]$ and let
$$I_6:=27(1+s^2)^4\sin^3(\theta)\cos(\theta)+s^3(p+4)^3(2p-s^2-3).$$We
differentiate $I_6$ with respect to $p$ and we obtain
$$\frac{\partial I_6}{\partial
p}=27(1+s^2)^4\theta_p(4\sin^2(\theta)\cos^2(\theta)-\sin^4(\theta))-s^3(p+4)^2(3s^2-8p+1).$$In
case that $\theta\in [\frac{2\pi}{3},\pi]$ (which implies that
$s\leq1/\sqrt{3}$), both summands are non-negative and therefore
$\frac{\partial I_6}{\partial p}\geq 0$. This implies that $I_6\leq
I_6|_{p=2}=0$.

For the case in which $\theta\in [0,\frac{2\pi}{3})$, we use Lemma
\ref{lem.interm} and we obtain that
\begin{equation*}
\begin{split}
I_1+I_4&=27(1+s^2)^4\sin(2\theta)\sin^2(\theta)+2s^3(p+4)^3(2p-s^2-3)\\
&<18(p+4)s(1+s^2)^3\sin^2\theta+2s^3(p+4)^3(2p-s^2-3),
\end{split}
\end{equation*}
hence we will study the sign of
$$
{\overline I}_6:=9(1+s^2)^3\sin^2\theta+s^2(p+4)^2(2p-s^2-3).
$$
As in the previous Lemma \ref{lem.T1T3}, we differentiate several
times with respect to $p$, and obtain
$$
\frac{\partial {\overline I}_6}{\partial
p}=9(1+s^2)^3\sin(2\theta)\theta_p+2s^2(p+4)(3p+1-s^2),
$$
$$
\frac{\partial^2 {\overline I}_6}{\partial
p^2}=18(1+s^2)^3\cos(2\theta)\theta_p^2+2s^2(6p+13-s^2),
$$
$$
\frac{\partial^3 {\overline I}_6}{\partial
p^3}=12\left[s^2-3(1+s^2)^3\sin(2\theta)\theta_p^3\right],
$$and finally,
$$\frac{\partial^4 {\overline I}_6}{\partial
p^4}=-\frac{8}{9}(1+s^2)\cos(2\theta)\arctan^4\left(\frac{1}{s}\right).$$
Therefore $$\frac{\partial^3 {\overline I}_6}{\partial
p^3}\geq\left\{\begin{array}
  {cc} \frac{\partial^3 {\overline I}_6}{\partial
p^3}|_{p=2}, & {\rm if\ } \theta(s,2)\leq \frac{\pi}{4}, \\ \\
\frac{\partial^3 {\overline I}_6}{\partial
p^3}|_{\theta=\frac{\pi}{4}}, & {\rm if }\ \theta(s,1)\leq
\frac{\pi}{4}\leq \theta(s,2), \\ \\ \frac{\partial^3 {\overline
I}_6}{\partial p^3}|_{p=1}, & {\rm otherwise.}
\end{array}\right.$$ In the first case,
$$\frac{\partial^3 {\overline I}_6}{\partial
p^3}\geq \frac{16}{3}s(1-s^4)\arctan^3\left(\frac{1}{s}\right) +
12s^2\geq \frac{4}{3s^2}(5s^4+4)>0,$$ the intermediate inequality
following after some algebraic manipulations from the standard fact
that $\arctan(1/s)\leq1/s$ for any $s>0$.

In the second and third case, we can estimate  
\begin{equation}\label{interm51}\frac{\partial^3 {\overline I}_6}{\partial p^3}\geq 
12s^2 - \frac{4}{3}(1+s^2)^3\arctan^3\left(\frac{1}{s}\right).
\end{equation}

It remains to show that the right hand side in \eqref{interm51} is
positive within the range
$$\frac{1}{\sqrt{3}}<s<\frac{1}{\tan(\frac{\pi}{8})}.$$ In order to
prove it, we introduce
$$l(s):=12s^2 -
\frac{4}{3}(1+s^2)^3\arctan^3\left(\frac{1}{s}\right).$$ We
differentiate it and, taking into account again that
$\arctan(1/s)\leq1/s$ we obtain
\begin{equation*}\begin{split}l'(s)& =4\left[(1+s^2)^2\arctan^2\left(\frac{1}{s}\right)\left(1-2s\arctan\left(\frac{1}{s}\right)\right)+6s\right]\\ &\geq
4\left[6s-(1+s^2)^2\arctan^2\left(\frac{1}{s}\right)\right]=:\tilde
l(s).\end{split}\end{equation*} and furthermore
\begin{equation*}\begin{split}\tilde
l'(s)& =8\left[(1+s^2)\arctan\left(\frac{1}{s}\right)\left(1-2s\arctan\left(\frac{1}{s}\right)\right)+3\right]\\ & >8\left[3-(1+s^2)\arctan\left(\frac{1}{s}\right)\right]
>0,\end{split}\end{equation*} as it is easy to check in the range $1/\sqrt{3}<s<1/\tan(\pi/8)$.
Then, $\tilde l$ increases with $s$ which implies that
$l'(s)> \tilde l(1/\sqrt{3})>0$; hence,
$l$ increases with $s$ yielding
$l(s)>l(1/\sqrt{3})=4-\frac{256\pi^3}{2187}>0$.

Hence $\frac{\partial^2 {\overline I}_6}{\partial p^2}$ is increasing with
respect to $p$. Thus, letting $p\to1$, we get
$$
\frac{\partial^2 {\overline I}_6}{\partial
p^2}\geq2\cos(2\theta)\arctan^2\left(\frac{1}{s}\right)(1+s^2)^3+2s^2(19-s^2).
$$
We notice that the right-hand side above is clearly positive when
$\cos(2\theta)>0$, whence it remains to prove it when
$\theta\in(\pi/4,2\pi/3)$, that is (taking into account that $p=1$)
$$
\frac{1}{\sqrt{3}}<s<\frac{1}{\tan(3\pi/20)}.
$$
In this range,
\begin{equation*}
\begin{split}
\frac{\partial^2 {\overline I}_6}{\partial
p^2}&\geq2\left[s^2(19-s^2)-(1+s^2)^3\arctan^2\left(\frac{1}{s}\right)\right].
\end{split}
\end{equation*}
Now, if $\frac{1}{\sqrt{3}}<s<1$, we use the inequality
$\arctan(1/s)\leq \frac{3s}{1+s^2}$ and we obtain that
$$\frac{\partial^2 {\overline I}_6}{\partial p^2}\geq
20s^2(1-s^2)>0.$$ If instead $s>1$ we use $\arctan(1/s)\leq1/s$ and
we obtain that $$\frac{\partial^2 {\overline I}_6}{\partial p^2}\geq
\frac{2}{s^2}[s^4(19-s^2)-(1+s^2)^3]>0.$$ for $1\leq
s\leq1/\tan(3\pi/20)$.

It follows that $\frac{\partial^2 {\overline I}_6}{\partial
p^2}>0$ and the first derivative is increasing with respect to $p$.
We again let $p\to1$ to get
\begin{equation}\label{interm16}
\frac{\partial {\overline I}_6}{\partial
p}\geq3(1+s^2)^3\sin\left(\frac{10}{3}\arctan\left(\frac{1}{s}\right)\right)\arctan\left(\frac{1}{s}\right)+10s^2(4-s^2)
\end{equation}
To prove that the right-hand side above is positive, we divide the
interval for $s$ as follows
$$
\left(\frac{1}{\sqrt{3}},\infty\right)=\left(\frac{1}{\sqrt
3},\frac{1}{\tan(3\pi/10)}\right)\cup\left[\frac{1}{\tan(3\pi/10)},\frac{1}{\tan(\frac{3\pi}{20})}\right]
\cup\left(\frac{1}{\tan(\frac{3\pi}{20})},\infty\right)=J_1\cup
J_2\cup J_3,
$$
and notice that for $s\in J_2$, both terms in the right-hand side of
\eqref{interm16} are positive. Also for $s\in J_1$, we easily get
\begin{equation*}
\begin{split}
\frac{\partial {\overline I}_6}{\partial
p}&\geq10s^2(4-s^2)-3(1+s^2)^3\frac{3\pi}{10}\\
&\geq\frac{10}{3}\left(4-\frac{1}{\tan^2(3\pi/10)}\right)-\frac{9\pi}{10}\left(1+\frac{1}{\tan^2(3\pi/10)}\right)^3>0
\end{split}
\end{equation*}
Finally, for
$s>1/\tan(3\pi/20)$, we can invert the sine function and the
inequality to prove is equivalent to
$$
\psi(s):=\frac{10}{3}\arctan\left(\frac{1}{s}\right)-\arcsin\left(\frac{10s^2(s^2-4)}{3(1+s^2)^3\arctan(1/s)}\right)>0.
$$
We show that $\psi$ is decreasing with respect to $s$, or
equivalently $\psi'(s)<0$. This is, after straightforward
calculations, equivalent to
$$
3(s^4-16s^2+8)-\arctan\left(\frac{1}{s}\right)\sqrt{9(s^2+1)^6\arctan^2\left(\frac{1}{s}\right)-100s^4(s^2-4)^2}<0,
$$
which can be proved by estimating $\arctan(1/s)$ in the convenient
way at each point, using the well-known inequalities
$$
\frac{1}{s}-\frac{1}{3s^3}<\arctan\left(\frac{1}{s}\right)<\frac{1}{s},
$$
we omit the details. It follows that $\psi$ is decreasing, whence it
is positive for $s\in(2,\infty)$, since
$\lim\limits_{s\to\infty}\psi(s)=0$.

We finally obtain that $\frac{\partial {\overline I}_6}{\partial p}>0$ for any
$s>\frac{1}{\sqrt{3}}$ and $p\in(1,2)$, hence ${\overline I}_6$ is increasing with
$p\in(1,2)$. The conclusion follows then from the fact that
${\overline I}_6|_{p=2}\equiv0$.
\end{proof}
\begin{lemma}\label{lem.T2T4}
Under the notation and conditions above, $I_2+I_4<0$ for any
$p\in(1,2)$ and $s\in(0,\sqrt{3/7})$.
\end{lemma}
\begin{proof}
The conclusion is equivalent to
$$
I_7:=3(3-p)(1+s^2)^2\sin(2\theta)+2s(p+4)(2p-3-s^2)<0.
$$
We differentiate it with respect to $p$:
$$
\frac{\partial I_7}{\partial
p}=-3(1+s^2)^2\sin(2\theta)+(3-p)(1+s^2)^2\cos(2\theta)\theta_p+2s(4p+5-s^2).
$$
Since $s\in(0,\sqrt{3/7})$, we have
$$
\theta=\frac{p+4}{3}\arctan\left(\frac{1}{s}\right)>\frac{5}{3}\arctan\left(\sqrt{\frac{7}{3}}\right)=:\theta_0>\pi/2,
$$
and if $\theta\in[3\pi/4,\pi]$, the first derivative is obviously
positive since $\cos(2\theta)>0$, $\sin(2\theta)<0$ and
$4p+5-s^2>9-3/7>0$. In the complementary case
$\theta\in(\theta_0,3\pi/4)$, we differentiate again with respect to
$p$ to find
$$
\frac{\partial^2I_7}{\partial
p^2}=-12(1+s^2)^2\left[\cos(2\theta)+(3-p)\sin(2\theta)\theta_p\right]\theta_p+8s>0,
$$
since $\sin(2\theta)<0$ and $\cos(2\theta)<0$ in the considered
range of $\theta$. Hence $\frac{\partial I_7}{\partial p}$ increases
with $p$ for all angles $\theta\in(\theta_0,\pi)$, in particular
attaining its minimum as $p\to1$. This implies that
\begin{equation}\label{interm17}\begin{split}
\frac{\partial I_7}{\partial
p}& \geq-3(1+s^2)^2\sin\left(\frac{10}{3}\arctan\left(\frac{1}{s}\right)\right)\\ & +4(1+s^2)^2\cos\left(\frac{10}{3}\arctan\left(\frac{1}{s}\right)\right)\arctan\left(\frac{1}{s}\right)+2s(9-s^2).
\end{split}\end{equation}
In order to prove that the right-hand side in \eqref{interm17} is
positive, we first notice that this is trivial for
$\theta_s:=\frac{5}{3}\arctan(1/s)\in[3\pi/4,\pi]$, since $\sin(2\theta_s)<0$ and
$\cos(2\theta_s)>0$. For $\theta_s\in(\theta_0,3\pi/4)$, we drop the
first term and prove the stronger fact that
$$
4(1+s^2)^2\cos\left(\frac{10}{3}\arctan\left(\frac{1}{s}\right)\right)\arctan\left(\frac{1}{s}\right)+2s(9-s^2)>0,
$$
or equivalently
$$
\frac{s(9-s^2)}{2(s^2+1)^2\arctan(1/s)}+\cos\left(\frac{10}{3}\arctan\left(\frac{1}{s}\right)\right)>0.
$$
It is easy to check that there exists a unique
$s_0\in(0,\sqrt{3/7})$ such that
$$
\frac{s(9-s^2)}{2(s^2+1)^2\arctan(1/s)}\geq1 \ \hbox{for} \ s\geq
s_0, \quad \frac{s(9-s^2)}{2(s^2+1)^2\arctan(1/s)}<1 \ \hbox{for} \
s<s_0.
$$
In the interval $[s_0,\sqrt{3/7})$, we are done. In the remaining
part $s\in(0,s_0)$, we can invert the cosine function, in a similar
manner as in proof of Lemma \ref{lem.T1T4}, and then prove that the
resulting function is decreasing by differentiation and estimates.
We skip the technical details.

Thus, $I_7$ is increasing with $p\in(1,2)$ and thus it is negative,
since $I_7|_{p=2}\equiv0$.
\end{proof}
From Lemmas \ref{lem.T1T3}, \ref{lem.T2T3}, \ref{lem.T1T4} and
\ref{lem.T2T4} we deduce that $I_1+I_2+I_3+I_4<C_0(p)<0$ for any
$p\in(1,2)$ fixed, $s\in(0,\infty)$, which, together with the good
behavior as $s\to\infty$, end the proof.
\end{proof}

\noindent{\bf Acknowledgements}. The authors have been partially
supported by the Spanish MEC project MTM2012-31103.

\bibliographystyle{plain}

\end{document}